\documentclass[12pt]{article}
\usepackage{amssymb,amsmath,graphicx}
\usepackage{theorem}
\usepackage[usenames]{color}
\usepackage{hyperref}

\date{October 11, 2010}

\def\bbR{{\mathbb R}}
\def\bbZ{{\mathbb Z}}

\def\sump{\!\!\mathop{\,\,{\sum}^{\,\prime}}}

\def\cC{{\cal C}}

\def\cP{{\cal P}}
\def\cS{{\cal S}}

\def\bW{{\bold W}}

\def\bV{{\bold V}}
\def\bc{{\bold c}}
\def\bE{{\bold E}}

\def\eps{\varepsilon}

\def\LA{\Lambda}
\def\La{\boldsymbol\Lambda}
\def\Om{{\Omega}}

\def\t{\tau}

\def\ttau{\tilde{\tau}}

\def\var{{\rm Var}}
\def\E{{\cal E}}
\def\cC{{\cal C}}
\def\bss{{\boldsymbol \sigma}}

\def\Int{{\text{\bf Int}\,}}
\def\Ext{{\text{\bf Ext}\,}}

\def\diam{{\text{\rm diam}}}
\def\dist{{\text{\rm dist}}}
\newcommand{\half}{\mbox{$\frac{1}{2}$}}
\renewcommand{\mod}{\text{ {\rm mod} }}
\def\ext{{\text{\rm ext}}}
\def\ord{{\text{\rm ord}}}
\def\dis{{\text{\rm dis}}}
\def\tun{{\text{\rm tunnel}}}

\def\sep{{\text{\rm sep}}}
\def\PW{\omega}

\newtheorem{theorem}{Theorem}[section]

\newtheorem{lemma}[theorem]{Lemma}

\newtheorem{corollary}[theorem]{Corollary}
\newtheorem{definition}[theorem]{Definition}
\numberwithin{equation}{section}

\def\proofof#1{\noindent{\bf Proof of #1}.}
\def\qed{\hfill$\Box$}

\begin{document}

\title{Tight Bounds for Mixing of the Swendsen-Wang
Algorithm at the Potts Transition Point}

\author{Christian Borgs%
\thanks{Microsoft Research, 1 Memorial Drive, Cambridge, MA 02124,
\{borgs, jchayes\}@microsoft.com}%
\and Jennifer T. Chayes$^*$ \and Prasad Tetali\thanks{School of
Mathematics and School of Computer Science, Georgia Tech,
Atlanta, GA 30332-0160, tetali@math.gatech.edu; research
supported in part by the NSF Grants DMS--9800351, DMS--0401239,
DMS--0701043.} }

\maketitle

\thispagestyle{empty}

\begin{abstract}
We study two widely used algorithms for the Potts model on
rectangular subsets of the hypercubic lattice  $\bbZ^d$
 -- heat bath dynamics
and the Swendsen-Wang algorithm --  and prove that, under
certain circumstances, the mixing in these algorithms is {\em
torpid} or slow.  In particular, we show that for heat bath
dynamics throughout the region of phase coexistence, and for
the Swendsen-Wang algorithm at the transition point, the mixing
time in a box of side length $L$ with periodic boundary
conditions has upper and lower bounds which are exponential in
$L^{d-1}$. This work provides the first upper bound of this
form for the Swendsen-Wang algorithm, and gives lower bounds
for both algorithms which significantly improve the previous
lower bounds that were exponential in $L/(\log L)^2$.
\end{abstract}

\section{Introduction}
\label{intro}

Convergence to equilibrium of heat bath dynamics and other
dynamics for several lattice spin models of statistical
mechanics has been of significant interest, for well over a
decade, in probability theory, statistical physics,
combinatorics and theoretical computer science.  While the
excellent monograph \cite{Fabio}
provides a testament to this, many exciting new results and new
techniques have since been developed. Fine examples of this
development include, on the fast mixing front, results for
Glauber dynamics on trees \cite{BKMP}, for Swendsen-Wang
algorithm on various classes of graphs \cite{CF}, for a simple
random walk on the super critical percolation cluster
\cite{BM,FR}. On the slow mixing side, results on the
Swendsen-Wang for the Potts model on the complete graph
\cite{GJ}, the heat bath algorithm for the Ising model at low
temperature \cite{Thom}, and on quasi-local algorithms for the
hardcore lattice gas model at low temperature
\cite{confversion} form similarly interesting and technically
challenging examples.

 In this paper, we study two
Monte Carlo Markov chains (MCMC), heat bath dynamics and the
empirically more rapid Swendsen-Wang algorithm, for the
$q$-state
 Potts model. Our work is a continuation of the work
begun some time ago in collaboration with several other authors
\cite{confversion}, where we obtained weaker bounds than those
we establish here.

The point of our previous work was to relate the mixing times
of MCMC in several models, including the Potts model,
 to the phase structure of the underlying equilibrium models.
 The Potts model is known
to undergo a phase transition from a so-called disordered phase
with a unique equilibrium state to an ordered phase with the
coexistence of multiple equilibrium states. In our previous
work we showed that, for the $q$-state Potts model on
rectangular subsets of the hypercubic lattice $\bbZ^d$ with
periodic boundary conditions, heat bath dynamics is slow or
{\em torpid} throughout the region of phase coexistence, while
the Swendsen-Wang algorithm is torpid at the transition point,
provided that $q$ is large enough. There the lower bounds on
the mixing time in a box of side length $L$ with periodic
boundary conditions were exponential in $L/(\log L)^2$.  In
this paper, we show that the mixing is even slower, obtaining
essentially optimal results: both lower and upper bounds on the
mixing time which are exponential in $L^{d-1}$.

Slowness of the Swenden-Wang algorithm for the Potts model at
the transition point was proved first on the complete graph
\cite{GJ}. This result initially came as a surprise to many
physicists who had tacitly assumed that the algorithm was fast
at all temperatures. Our previous work \cite{confversion} was
the first to establish such a result on subsets of the
 hypercubic lattice, a case which is both more physically
relevant and technically much more challenging than the
complete graph.   To overcome these difficulties, we used some
deep results from mathematical physics, which we now extend. In
particular, our work brings to bear and extends, statistical
physics expansion techniques for the problem of controlling the
number of cutsets in graphical expansions of these models.
Specifically, we use the so-called Pirogov-Sinai theory
\cite{PS} from the statistical physics literature, in the form
adapted to the Potts model by Borgs, Koteck\'y and Miracle-Sole
(\cite{BK}, \cite{BKM}). We also use the  isoperimetric
inequalities of Bollob\'as and Leader \cite{iso}, as well as a
large deviations technique borrowed from \cite{B}.

For Markov chains that  change the value of only  a bounded
number of spins -- such as the heat bath algorithm, it is easy
to obtain upper bounds exponential in $L^{d-1}$ using either
refined canonical path arguments as in \cite{BKMP} or recursive
bounds on the Dirichlet form as in \cite{cesi}, see
\cite{Fabio} for a review.  However, for the Swendsen-Wang
algorithm, which is highly non-local in the spin
representation, such an upper bound is not obvious. However, it
turns out that a refinement of the bounds on the Dirichlet form
can be used to obtain the desired upper bound.  More generally,
introducing a new graph parameter which we call
the``decomposition width,'' we derive upper bounds for
Swendsen-Wang on arbitrary graphs, which as a special case
proves that Swendsen-Wang on trees is polynomial in the number
of vertices for all temperatures.

The real challenge is to obtain a {\it lower} bound which is
exponential in $L^{d-1}$, significantly improving the
{$e^{cL/\log^2 L}$} lower bound of \cite{confversion}.
While the previous bound required that we consider only
contours which can be embedded into $\bbZ^d$, this optimal
lower bound requires that we deal explicitly with the topology
of the torus.  In particular, we must distinguish between
surfaces with vanishing and non-vanishing winding numbers,
which we call contours and interfaces respectively.  Moreover,
in order to deal optimally with the contours, we need to define
an appropriate notion of exterior and interior which allows us
to establish a partial order on the set of contours.  This in
turn is used to develop the appropriate Pirogov-Sinai theory on
the torus.

In order to state our results precisely, we  need a few
definitions. Let $G=(V,E)$ be a finite graph and let $\beta>0$.
For a positive integer $q$, let $[q]=\{1,2,\ldots ,q\}$. The
Gibbs measure of the (ferromagnetic) $q$-state Potts model on
$G$ at inverse temperature $\beta$ is a measure on $[q]^V$ with
density
\begin{equation}\label{Mudef}
\mu_G(\bss) = \frac{e^{-\beta H_G(\bss)}}{Z_G},
\end{equation}
where $\bss \in [q]^V$ is a spin configuration. Here
\begin{equation}\label{Hdef}
 H_G(\bss) = \sum_{xy \in E}\left(1-\delta(\sigma_x, \sigma_y)\right)
\end{equation}
is the Hamiltonian, and the normalization factor
\begin{equation}\label{Zdef}
Z_G = \sum_{\bss \in [q]^V} e^{- \beta H_G(\bss)}
\end{equation}
is the partition function. (In the above, $\delta$ denotes the
Kronecker delta function.)

For a finite $\LA \subset \bbZ^d$ we define the measure
$\mu_{\LA,1}$ with the ``1-boundary condition'' by setting all
the spins at the external boundary of $\LA$ to 1. Explicitly,
for $\bss \in \{1,2,\ldots, q\}^{\LA}$, let
\[H_{\LA,1}(\bss) = H_{G[\LA]}(\bss) +
\sum_{{x\in \LA} \atop{{y\in \LA^c} \atop{|x-y|=1}}}
(1-\delta(\sigma_x,1))\,,\] where $G[\LA]$ is the induced subgraph of
$\LA$ and $|\cdot|$ is the $l_1$ distance.

Then $\mu_{\LA,1}$ is the probability measure with density
\begin{equation}
\label{mu1bd}
\mu_{\LA,1}(\bss)
= \frac{e^{-\beta H_{\LA,1}(\bss)}}{Z_{\LA,1}},
\end{equation}
where $Z_{\LA,1}$ is the appropriate normalization factor.

{The infinite volume  magnetization is defined as
$M(\beta)=\lim_{L\to\infty} M_{\LA_L}(\beta)$ where
$\LA_L=\{1,\dots,L\}^d$ and
\[ M_\LA(\beta) =
\frac 1{|\LA|}\sum_{x\in\LA}
\Bigl(\mu_{\LA,1}(\{\sigma_x=1\})
- \frac{1}{q}\Bigr).
\]
 By standard correlation
inequalities, the limit $\lim_{L\to\infty} M_{\LA_L}(\beta)$}
is known to exist and to be monotone nondecreasing in $\beta$.
The transition point is defined by
\[\beta_0 = \beta_0(\bbZ^d)= \inf \{\beta : M(\beta) > 0\}.\]

{Let $d\geq 2$.  Then $0<\beta_0<\infty$, and the model
has a unique (infinite-volume) Gibbs state for $\beta <
\beta_0$, and  at least $q$ extremal translation-invariant
Gibbs states for $\beta > \beta_0$ \cite{ACCN}.  For $q$ small
enough (depending on $d$), the model is believed to have a
unique Gibbs state at $\beta_0$, whereas for $q$ large enough
(again depending on $d$), it is known \cite{KS,LMMRS} to have
 $q + 1$ extremal
translation-invariant Gibbs states at $\beta_0$, with
\begin{equation}
\label{beta0-asym}
\beta_0 =\frac 1d\log q+O(q^{-1/d}).
\end{equation}

Next} we define the mixing time of a finite Markov chain with
state space $\Omega$. Let $P$ denote the transition probability
matrix of an irreducible Markov chain with the (unique)
stationary measure $\mu$. The (variational) mixing time of such
a chain is defined as
\begin{equation}
\label{t-mix} \tau = \min\left\{t: d(t) \le \frac{1}{2e} \right\}\,,
\end{equation}
where
\[ d(t) = \max_{\bss \in \Omega}
\max_{A \subset \Omega} \Big|\mu(A) -
\sum_{\bss' \in A} P^t(\bss,\bss')\Big|. \]
In this paper, we will consider several Markov chains for the
Potts model on the torus
\[
T_{L,d} = \left(\bbZ/L \bbZ \right)^d.
\]
The chains we consider are the heat bath (or Glauber) dynamics,
and the presumably much faster SW algorithm; see
Section~\ref{Alg-sec} of the definition of these chains. We
denote the mixing times of these algorithms for the $q$-state
Potts model on the torus $T_{L,d}$  by $\tau^{\rm HB} =
\tau^{\rm HB} (T_{L,d})$ and $\tau^{\rm SW} = \tau^{\rm SW}
(T_{L,d})$.

\begin{theorem}
\label{ub-thm} There are universal constants $k_1,k_2<\infty$
such that, for $\beta>0$, $d\geq 2$ and any positive integer
$L$, the following bounds hold
\begin{align}
\tau^{\rm HB}(T_{L,d}) &\le e^{(k_1+k_2\beta) L^{d-1}}
\label{HB-tau-up}
\\
\tau^{\rm SW}(T_{L,d}) &\le e^{(k_1+k_2\beta) L^{d-1}}.
\label{SW-tau-up}
\end{align}

\end{theorem}

 In order to prove this, in Section~\ref{UB-sec},  we introduce the
``partition width'' of a graph, a notion that  may be of
independent interest. As a corollary of our proof, we also
obtain that the mixing time of SW on a tree with $n$ vertices,
maximum degree $d_{\max}$, and depth $O(\log n)$ is bounded by
$n^{1+O(\beta d_{\max})}$, see
 Corollary~\ref{sw_ub-cor} for the precise statement.
This generalizes the  result of \cite{BKMP}, which gives
polynomial mixing for the HB algorithm on trees, to the SW
algorithm.

\begin{theorem}\label{thm:SW}
Let $d\ge 2$.  Then there exists a constant $k_3=k_3(d)>0$ such
that, for $q$ and $L$ sufficiently large, the following bounds
hold:
\begin{align}
\tau^{\rm HB}(T_{L,d})
&\ge e^{k_3\beta L^{d-1}} \qquad \mbox{ for all }
\beta \ge \beta_0(\bbZ^d)
\label{HB-tau-low}
\\
\tau^{\rm SW}(T_{L,d})
&\ge e^{k_3\beta L^{d-1}} \qquad \mbox{ for  }
\beta = \beta_0(\bbZ^d).
\label{SW-tau-low}
\end{align}
\end{theorem}

Very roughly speaking, the reason for the heat bath lower bound
is that this algorithm cannot move quickly among (the finite
analogs of) the $q$ translation-invariant extremal states
present for $\beta
> \beta_0$. On the other hand, the Swendsen-Wang algorithm can move
quickly among these $q$ states, but cannot move quickly between
these $q$ states and the one additional  translation-invariant
extremal state present at $\beta = \beta_0$. Whenever the
algorithm cannot move quickly between  states, in order for the
system to mix, it must pass through a configuration with a
``separating surface'' of size at least $L^{d-1}$ between the
relevant states.


 The organization of this paper is as follows.  In
Section~\ref{sec:prelim}, we  define the algorithms and the
necessary notions from the theory of MCMC.  In
Section~\ref{UB-sec}, we introduce the notion of partition
width and establish the upper bound on the mixing time.  In
order to obtain the corresponding lower bound, we need some
preparation: in Section~\ref{sec:Cont-Rep} we  construct the
contour representation of the model, while
Section~\ref{sec:geom-cont} establishes the required geometric
properties of contours and interfaces. This section can perhaps
be skipped on first reading. In the next section, we state the
necessary bounds from Pirogov-Sinai (to be proved in the
appendix), and use these to establish two key estimates needed
for the main proof: a bound on the probability of interfaces,
and a suitable large deviation bound.  Using these bounds, we
then prove our main result in Section \ref{sec:LowerBound}.

The reader only interested in the main flavor of our proofs
should perhaps start with Section~\ref{sec:proof-strat}, where
we explain the main proof strategy, and then immediately jump
to Section~\ref{sec:LowerBound}, glancing back at
Section~\ref{sec:Cont-Rep} and Section~\ref{decomp} as
necessary.

\section{MCMC Preliminaries}
\label{sec:prelim}

\subsection{Algorithms for the Potts Model}

\label{Alg-sec}

There are several MCMC algorithms that are used to generate a
random sample from the distribution corresponding to the
ferromagnetic Potts model. The heat bath is perhaps the
simplest such Markov chain.  Its transitions are as follows:
Choose a vertex at random, and modify the spin of that vertex
by choosing from the distribution conditional on the spins of
the other vertices remaining the same. In contrast, the
Swendsen-Wang algorithm can alter the spins on many vertices in
each iteration.

Throughout this section $G=(V,E)$ is a fixed finite graph. For
a subgraph $\tilde{G}=(\tilde{V},\tilde{E})$ of $G$, we denote
the set of (connected) components of $\tilde{G}$ by
$\cC(\tilde{G}) = \cC(\tilde{V},\tilde{E})$, and  its
cardinality by $c(\tilde{G}) = c(\tilde{V},\tilde{E})$.
Finally, for a spin configuration $\bss \in [q]^V$, let
$E(\bss)$ be the set of ``monochromatic edges'' $xy\in E$ with
$\sigma_x=\sigma_y$.

\

\noindent {\bf Heat Bath:} From a spin configuration $\bss \in
[q]^V$, we construct a new configuration $\bss'$ as follows:
\begin{description}
\item[HB1] Choose $v$ uniformly at random from $V$.
\item[HB2] Take $\sigma'_w = \sigma_w$, for all $w \in
    V\setminus \{v\}$, and change $\sigma_v$ to $\sigma'_v$
    with probability
$$
{\mu_G\left(\sigma'_v \big| \bss_{V\setminus \{v\}}\right)} =
\frac{\displaystyle \exp\Bigl\{\beta \sum_{w\in V \atop{vw \in E}}
\delta(\sigma'_v,\sigma_w) \Bigr\}} {\displaystyle \sum_{k=1}^q
\exp \Bigl\{\beta \sum_{w\in V \atop{vw\in E}} \delta(k,\sigma_w)
\Bigr\}}.
$$
For future reference we denote the transition matrix of
this chain by
\[ P_G^{{\rm HB}}(\bss,\bss')
= \frac{1}{|V|} \sum_{v\in V}\Bigl(
\mu_G\left(\sigma'_v \big| \bss_{V\setminus \{v\}}\right)
\prod_{w\neq v} \delta(\sigma'_w,\sigma_w)\Bigr).\]

\end{description}

In practice, an alternative method, the Swendsen-Wang algorithm
\cite{SW}, is often preferred.

\

\noindent {\bf Swendsen-Wang Algorithm:} For $\bss \in [q]^V$:

\begin{description}
\item[SW1] Let $E(\bss) \subset E$  be the set  of
    monochromatic edges. Delete each edge of $E(\bss)$
    independently with probability $1-p$, where $p=
    1-e^{-\beta}$. This gives a random subset $A \subset
    E(\bss)$.
\item[SW2] The graph $(V,A)$ consists of connected
    components. For each component, choose a color (spin)
    $k$ uniformly at random from $[q]$, and for all
    vertices $v$ within that component, set $\sigma'_v=k$.

\end{description}

Again for future reference, we denote the transition matrix of
this chain by
\[P_G^{{\rm SW}}(\bss,\bss')
= \sum_{A\subset E(\bss)} p^{|A|} (1-p)^{|E(\bss)\setminus A|}
\prod_{C\in \cC(V,A)} \Bigl(\frac{1}{q} \sum_{k=1}^q \prod_{v\in
V(C)} \delta(\sigma'_v, k)\Bigr).\]

The Swendsen-Wang algorithm was motivated by the equivalence of
the ferromagnetic $q$-state Potts model and the {\em random
cluster} model of Fortuin and Kasteleyn \cite{FK}, which we now
describe. Fortuin and Kasteleyn realized that the Potts model
partition function $Z_G$ and expectations with respect to the
measure $\mu_G$ can be rewritten in terms of a weighted graph
model on spanning subgraphs $(V,A)\subset G$ with weights
\begin{equation}
\label{FKmes} \nu_G(A) = \frac 1{Z_G} p^{|A|}(1-p)^{|E \setminus A|}
q^{c(V,A)}.
\end{equation}

The relationship between the two models is elucidated in a
paper by Edwards and Sokal \cite{ES}. The Potts and random
cluster models are defined on a joint probability space
$[q]^V\times 2^E$. The joint probability $\pi({\bss},A)$ is
defined by
\begin{equation}
\label{ESmes} \pi_G({\bss},A)= {1\over Z_G} p^{|A|}(1-p)^{|E
\setminus A|}\prod_{xy\in A} \delta(\sigma_x,\sigma_y).
\end{equation}
By summing over ${\bss}$ or $A$ we see that the marginal
distributions are $\nu_G$ or $\mu_G$ respectively.

A step ${\bss} \to {\bss}'$ of the Swendsen-Wang algorithm can
be seen as (i) choose a random $A'$ according to
$\pi_G({\bss},\cdot)/\mu_G(\bss)$ and then (ii) choose a random
${\bss}'$ according to $\pi(\cdot,A')/\nu_G(A')$. After Step
SW1,  we say that we are in the random cluster representation
of the chain.

\subsection{Mixing Time and Related Quantities}
\label{mixtime-sec} Throughout this section, let $P$ be the
transition matrix of an irreducible Markov chain on a finite
state space $\Omega$,  let $\mu$ be the stationary distribution
of $P$,
 i.e., $\mu(\bss') = \sum_{\bss \in \Om} \mu(\bss)
P(\bss, \bss')$, and let
$\mu_{\min}=\min_{\bss\in\Om}\mu(\bss)$.
  We denote the mixing time, defined in
(\ref{t-mix}), by $\tau(P)$, and define the inverse gap (or
eigentime) $\ttau(P)$ as

\begin{equation}
\label{gap-def}
\ttau(P) = \sup_g \frac{\var \ g}{\E(g,g)},
\end{equation}
where the supremum is over all real valued functions $g$ on
$\Omega$ with $\var \ g > 0$. Here as usual,
\[\var \ g = \var_{\mu} \ g = \frac{1}{2} \sum_{\bss,\bss'}
\left(g(\bss) - g(\bss') \right)^2 \mu(\bss)\mu(\bss'),
\]
and
\[\E(g,g) = \E_P(g,g) = \frac{1}{2} \sum_{\bss,\bss'}
\left(g(\bss) - g(\bss') \right)^2 \mu(\bss) P(\bss,\bss').\]
If $P$ is reversible, $\tilde\tau(P)$ is just
$(1-\beta_2(P))^{-1}$, where $\beta_2(P)$ is the second largest
eigenvalue of $P$.

It is well known that the inverse gap can be bounded above in
terms of the mixing times; if the chain is lazy, i.e., if the
minimal self-loop probability $\min_\bss P(\bss,\bss)$ is
uniformly bounded from below, a bound in the opposite direction
is also not very hard to prove, see, e.g., \cite{AF}.  However,
the SW chain is not lazy.  Instead of the standard upper bound
on $\tau(P)$ in terms of $\tilde\tau(P)$, we therefore use a
bound from \cite{MT06}. For reversible chains, this bound gives
\begin{equation}
\label{mix-gap-bd}
\tau(P) \le \ttau(P^2) \log\Bigl(\frac{e^2}{\mu_{\rm min}}\Bigr),
\end{equation}
where, as usual, $P^2(\bss,\bss')=\sum_{\bss''\in
\Om}P(\bss,\bss'')P(\bss'',\bss')$, denotes the transition
matrix of the two-step chain.

 We will also need an identity for the mixing time of a product chain.
Let $\Om_1, \Om_2$ be finite sets, and let $P_1, P_2$ be the
transition matrices of two irreducible Markov chains on $\Om_1$
and $\Om_2$ with stationary distributions $\mu_1$ and $\mu_2$
respectively.  Let $\Om_1\times\Om_2$ be the set of all pairs
$\bss=(\bss^{(1)},\bss^{(2)})$ with $\bss^{(1)}\in\Om_1$ and
$\bss^{(2)}\in\Om_2$. Then the product chain is defined as the
chain with the transition matrix
\begin{equation}
\label{prod-chain} (P_1\times P_2)\bigl(\bss,\tilde\bss\bigr)
 =
P_1(\bss^{(1)},\tilde\bss^{(1)})
P_2(\bss^{(2)},\tilde\bss^{(2)}).
\end{equation}
Let
\[ (\mu_1 \times \mu_2)(\bss)
= \mu_1(\bss^{(1)}) \mu_2(\bss^{(2)}).
 \]

If both $P_1$ and $P_2$ have non-negative eigenvalues, then
$P_1\times P_2$ has non-negative eigenvalues and
$\beta_2(P_1\times P_2)=\max\{\beta_2(P_1),\beta_2(P_2)\}$.
Using this fact, one immediately shows that $P_1\times P_2$ is
irreducible with stationary distribution $\mu_1\times \mu_2$
and obeys the bound
\begin{equation}
\label{prod-tau}
\tilde\tau(P_1 \times P_2) =
\max\bigl\{ \tilde\tau(P_1), \tilde\tau(P_2) \bigr\}.
\end{equation}

For our lower bounds, we use the notion of conductance and its
relation to the mixing time. Setting
\[Q(S,S') = \sum_{\bss \in S} \sum_{\bss' \in S'}
\mu(\bss) P(\bss, \bss'), \]
the conductance of a set $S\subset \Om$ can be defined as
\begin{equation}
\label{cond} \Phi_S =  \frac{Q(S,S^c)}{\mu(S)\mu(S^c)}.
\end{equation}
Finally the conductance of a Markov chain with the transition
matrix $P$ is
\begin{equation}
\label{conductance} \Phi(P) = \min_{S: 0<\mu(S)<1} \Phi_S.
\end{equation}

The mixing time can be easily bounded from below in terms of
the conductance, see \cite{B},\cite{CLP}, \cite{DFJ}:
\begin{equation}
\label{mix-cond-lb}
\t(P)\geq \frac{e-1}{e} \frac{1}{\Phi(P)}.
\end{equation}

\subsection{Proof Strategy}
\label{sec:proof-strat}

 In order to prove Theorem \ref{ub-thm}, we will want to
give an upper bound on the inverse gap $\ttau$ defined in
\eqref{gap-def}. To this end, it will be convenient to consider
the SW algorithm on a general graph $G$.  We then iteratively
partition the set $V$ of vertices of $G$ into two sets $V_1$
and $V_2$, and show that the inverse gap of $(P_G^{{\rm
SW}})^2$ is bounded by the inverse gap of the product chain for
SW on the two induced graphs $G[V_1]$ and $G[V_2]$, times a
factor which is exponential in the number of edges with one
endpoint in $V_1$ and one endpoint in $V_2$. With the help of
\eqref{prod-tau} and \eqref{mix-gap-bd}, this allows us to
bound the mixing time of SW by a number which is exponential in
a quantity we call the partition width of the graph $G$.
Applied to the torus $T_{L,d}$, this gives a bound which is
exponential in $L^{d-1}$, and applied to a tree, this will give
a bound which is polynomial in the number of vertices.  The
bound for the HB algorithm follows a similar strategy.

To prove Theorem~\ref{thm:SW}, we will use the lower bound
\eqref{mix-cond-lb}. To this end, we will want to find a set of
spin configurations $S$ such that $\Phi_S$ is exponentially
small in $L^{d-1}$. Recalling that the Potts model at the
transition point exhibits the coexistence of $q$ ordered phases
and one translation invariant phase, we will want to exploit
the fact that the SW algorithm cannot transition easily between
(the finite volume analogue of) the ordered phases and the
disordered phase.  To make this precise, we  define $S$ to be
the set $S=\{\bss : |E(\bss)| \ge (1-\alpha)dL^d\}$, where
$\alpha>0$ is a small constant, say $\alpha=1/3$. Thus $S$
consists of the configurations whose set of monochromatic edges
form  almost all of $E$.

For large $q$ the inverse transition temperature $\beta_0$ is
large as well, implying that at $\beta=\beta_0$, the
probability of deleting an edge in the first step of the SW
algorithms is small;  starting from a configuration $\bss\in S$
it is therefore unlikely that after one step of the Markov
chain, the new configuration $\bss'$ is such that the number of
edges in $E(\bss')$ is much smaller than $(1-\alpha)dL^d$. (The
probability that it is smaller than, say $\alpha d L^d$, is
actually exponentially small in $L^d$).  But it is also
unlikely that a configuration $\bss'\in S^c$ has  a number of
monochromatic edges which is larger than
 $\alpha dL^d$, since both requirements together imply that the number of edges
lies between $\alpha dL^d$ and  $(1-\alpha)dL^d$.  This
corresponds neither to an ordered phase (which would have more
than $(1-\alpha)dL^d$ monochromatic edges), nor to a disordered
phase (which would have less than $\alpha dL^d$ monochromatic
edges), and thus to a configuration which has low probability.
Thus with high probability a configuration
 $\bss\in S$ leads to a new configuration $\bss'$ which is in $S$ as well, showing
that $\phi_S$ is small.

To make this quantitative, we will have to show that the weight
of the configurations in $S_0 = \{\bss' : \alpha dL^d \le
|E(\bss')| \le (1-\alpha)L^d\}$ is exponentially small in
$L^{d-1}$.  To this end, we will first switch to the FK
representation \eqref{FKmes}, and then describe an edge
configuration $A\subset E$ in terms of a hypersurface
separating regions with  edges in $A$ from regions with edges
in $E\setminus A$.  We then decompose this hypersurface into
connected components, some of which to be called contours, and
others to be called interfaces.  While contours can have small
or large size, interfaces will always have size at least
$L^{d-1}$.

To prove our desired bounds on $\mu(S_0)$, we will show that
the probability of a configuration with an interface is
exponentially small in the size of the interface, leaving us
with the analysis of configurations without interfaces.  These
will come in two classes: configurations describing
perturbations of $A=E$ (we denote the set of these
configurations by $\Omega_\ord$), and configurations describing
perturbations of $A=\emptyset$ (to be denoted by
$\Omega_\dis$).  Our last step then consists of a large
deviations bound showing that with high probability, for
configurations $A\in \Omega_\ord$, the common exterior of a set
of contours has size at least $(1-\frac 12\alpha)L^d$, and
similarly for configurations $A\in\Omega_\dis$.  This will
imply that with high probability, the set of monochromatic
edges of an ordered configuration $\bss$ has at least
$(1-\alpha)d L^d$ edges, while that of a disordered
configuration has at most $\alpha d L^d$ edges. Put together,
these estimates give the desired bound for $\mu(S_0)$.

Our approach differs in several aspects from the approach taken
in \cite{confversion}, which led to a conductance bound that
was exponentially small in $L/(\log^2L)$.

First, the bounds in \cite{confversion} relied on a combination
of Pirogov-Sinai theory with the finite-size scaling theory
developed in \cite{BK} and \cite{BKM}, where contours were by
definition objects that could be embedded in $\bbR^d$.  This
allowed for an immediate application of standard Pirogov-Sinai
results, but  produced error terms that were only exponential
in $L$, which is not strong enough for our current purpose.  To
avoid this problem, we define contours in a purely topological
manner,  by requiring that their $\bbZ_2$ winding number with
respect to the torus is equal to zero.  This has the advantage
that the objects which cannot be classified as contours (we
call them interfaces) must have size at least $L^{d-1}$, since
all surfaces of smaller size have winding number zero. The
price we have to pay is that the set of contours now contains
objects like long tubes winding around the torus which cannot
be embedded into $\bbR^d$,  preventing us from applying the
standard Pirogov-Sinai machinery as more or less a black box.

Instead, we will show that Pirogov-Sinai theory does not really
rely on the topology of $\bbR^d$, but rather on the implied
structure of partial order  on contours.  More precisely, it
relies on the fact that for any set of pairwise non-overlapping
contours and interfaces, this partial order leads to a Hasse
diagram that is a forest -- this is expressed in
Lemma~\ref{lem:partial-order}, see also
Lemma~\ref{lem:int-ext-exist} and
Definition~\ref{def:Ext-gamma}.

Second, we will use a large deviation bound obtained by adding
an artificial magnetic field to the contour model, see
Section~\ref{sec:LargeDev} for details.  This turns out to be
much more efficient than the iterative, combinatorial arguments
from \cite{confversion}, allowing us to improve a bound that is
exponentially small in $L/(\log L)^2$ to one which is
exponentially small in $L^{d-1}$.

\section{Upper Bound on Mixing Time}
\label{UB-sec}

In order to prove our upper bound on the mixing time, it will
be convenient to  prove a more general theorem, involving a new
notion, called the ``partition width'' of a graph, which we
expect may be of independent interest. We need some notation.

Given a graph $G=(V,E)$, we define a hierarchical partition
$\cal P$ of $V$ by first dividing $V$ into two non-empty
subsets $V_1$, $V_2$, and then successively subdividing each
set with more than one element into two further subsets until
all sets contain only one element.  Each such partition can be
described by a rooted binary tree as follows: the vertices are
subsets of $V$, with the root being $V$, the leaves being the
singletons $\{x\}$, $x\in V$. In addition, we have the
constraint that for any vertex $V_i$ with children $V_{i1}$,
$V_{i2}$, we have $V_i=V_{i1}\cup V_{i2}$ and $V_{i1}\cap
V_{i2}=\emptyset$.  If $V_i$ has children $V_{i1}$ and
$V_{i2}$, we define the weight $w_{\cP}(V_i)$ of $V_i$ as the
number of edges between $V_{i1}$ and $V_{i2}$ in $G$; if $V_i$
is a leaf, we set its weight to zero.

For $x\in V$, we now define the separation cost $\sep_\cP(x)$
of $x$ as the sum of all vertex weights along the path from the
root to $x$.  The cost of a partition $\cP$ is then defined as
$\sep(\cP)=\max_{x\in V}\sep_\cP(x)$, and the
 {\em partition width}
of $G$ is defined as $\PW(G)=\min_\cP\sep(\cP)$.

\begin{theorem}
\label{sw_ub-thm} For any finite graph $G=(V,E)$, the mixing
time of the SW algorithm obeys the upper bound
\[\tau(P^{\rm SW}_{G}) \le e^{5\beta \PW(G)}
\bigl(2+|V|\log 2+\beta|E|\bigr).
\]
\end{theorem}

Before proving the theorem, we state (and prove) the following
corollary, which illustrates the usefulness of our notion of
partition width.

\begin{corollary}
\label{sw_ub-cor} Let $\LA$ be a rectilinear subset of
$\bbZ^d$, let $T_{L,d}$ be the $d$-dimensional torus of side
length $L$, and let $T$ be a tree on $n$ vertices with maximum
degree $d_{\max}$ and depth $c\log n$.  Then
\begin{align}
\tau(P^{\rm SW}_{\LA})
&\le e^{45 \beta |A(\LA)|}
\bigl(2+(\log 2+d\beta)|\La|\bigr),
\label{PW-box-bound}
\\
\tau(P^{\rm SW}_{T_{L,d}})
&\le e^{75\beta L^{d-1}}
\bigl(2+(\log 2+d\beta)L^d\bigr),
\\
\tau(P^{\rm SW}_{T})
&\le n^{5 c\beta d_{\max}}\bigl(2+(\log 2+\beta)n\bigr).
\end{align}
Here $A(\La)$ is the volume of $\La$ divided by the minimal
side length.
\end{corollary}
Note that the second bound of corollary implies the  bound
\eqref{SW-tau-up} in Theorem~\ref{ub-thm}. The bound
\eqref{HB-tau-up} of this theorem can either be proved by
generalizing Theorem~\ref{sw_ub-thm} to the heat bath algorithm
(the proof is actually easier for this case), or by using the
canonical path techniques of \cite{BKMP}.

\proofof{Corollary~\ref{sw_ub-cor}} Let $G=(V,E)$ be an
arbitrary finite graph, and let $V=V_1\cup V_2$ be a
decomposition of $V$ into two disjoint subsets.  Using the
definition of the partition width, one easily verifies that
\begin{equation}
\label{PW-ind-bd}
\PW(G)\leq |E_{12}| +\max\{\PW(G[V_1]),\PW(G[V_2])\}\,,
\end{equation}
where $E_{12}$ is the set of edges between $V_1$ and $V_2$.
Using this bound, it is easy to verify by induction, that for a
tree $T$ of maximal degree $d_{\max}$ and depth $D$, one has
\[\PW(T)\leq d_{\max} D,\]
which in turn gives the bound in the corollary for trees.

We are thus left with proving upper bounds for the partition
width for cubic subsets of $\bbZ^d$ and the torus $T_{L,d}$. We
start with a rectilinear subset of side-lengths $L_1\geq
L_2\geq\cdots\geq L_d$, which we denote by $[L_1,\dots,L_d]$.
To this end, we first note that
\[
\PW(G')\leq \PW(G)\,,
\]
whenever $G'$ is a spanning subgraph of $G$, implying that
$\PW([L_1,\dots,L_d])$ is non-decreasing in the side-lengths
$L_i$. Consider a set of sidelength $L_1\geq\dots\geq L_d$ with
$L_1\geq 2$. Using the bound \eqref{PW-ind-bd}, the
monotonicity of $\PW([L_1,\dots,L_d])$  and the fact that
$\lceil L_i/2\rceil\leq 2L_i/3$, whenever $L_i\geq 2$, we then
bound
\begin{equation}
\begin{aligned}
\PW\bigl(\bigl[L_1,\dots,L_d\bigr]\bigr)
&\leq
L_2\cdots L_d +
\PW\Bigl(\Bigl[\lceil \frac{L_1}2\rceil,L_2,\dots,L_d\Bigr]\Bigr)
\\
&\leq
L_2\cdots L_d + \lceil \frac{L_1}2\rceil L_3\cdots L_d +
\PW\Bigl(\Bigl[\lceil \frac{L_1}2\rceil,
\lceil\frac{L_2}2\rceil,\dots,L_d\Bigr]\Bigr)
\\
&\leq\dots
\\
&\leq
\biggl(1+\frac 23 +\Bigl(\frac 23\Bigr)^2+\cdots
\Bigl(\frac 23\Bigr)^{k-1}\biggr)L_1\cdots L_{d-1}
\\
&
\qquad
+\PW\Bigl(\Bigl[\lceil \frac{L_1}2\rceil,\lceil\frac{L_2}2\rceil,
\dots,\lceil \frac{L_d}2\rceil\Bigr]\Bigr)
\\
&\leq 3L_1\cdots L_{d-1}
+\PW\Bigl(\Bigl[\lceil \frac{L_1}2\rceil,\lceil\frac{L_2}2\rceil,
\dots,\lceil \frac{L_d}2\rceil\Bigr]\Bigr).
\end{aligned}
\end{equation}
where $k$ is the smallest $i$ such that $L_i\geq 2$. Using this
bound, it is now easy to prove by induction that
\[
\PW([L_1,\dots,L_d])\leq 9 A([L_1,\dots,L_d]),
\]
implying the desired bound for the inverse gap on rectilinear
sets.

Next, we would like to bound the partition width of the torus
$T_{L,d}$ using the just established bound for  rectilinear
sets. To this end we successively cut the torus in the $d$
different coordinate directions, proceeding as in the proof
above.  Here, however, since we have a torus rather than a
rectilinear set, we need two cuts rather than one cut in each
direction to obtain two components.  Keeping this in mind, we
get

\begin{equation}
\begin{aligned}
\PW\bigl(T_{L,d}\bigr)
&\leq
2\biggl(1+\frac 23 +\Bigl(\frac 23\Bigr)^2+\cdots
\Bigl(\frac 23\Bigr)^{d-1}\biggr)L_1\cdots L_{d-1}
\\
&
\qquad
+\PW\Bigl(\Bigl[\lceil \frac{L}2\rceil,\lceil\frac{L}2\rceil,
\dots,\lceil \frac{L}2\rceil\Bigr]\Bigr)
\\
&\leq 6L^{d-1}
+9L^{d-1}
=15L^{d-1},
\end{aligned}
\end{equation}
which implies the desired bound on the mixing time. \qed

The proof of Theorem~\ref{sw_ub-thm} is based on the following
lemma.

\begin{lemma}
\label{subgraph-lem} Let $G=(V,E)$, let $P_G=P^{\rm SW}_{G}$
and $G_0 = (V,E_0)$, where $E_0$ is an arbitrary subset of $E$.
Then
\[ \ttau(P^2_{G})
\le \ttau(P^2_{G_0}) \  e^{5\beta |E\setminus E_0|}.
\]
\end{lemma}

\begin{proof}
Recall that a single transition of the SW dynamics consists of
two steps. Given a current Potts configuration $\bss$, the
first step identifies connected components of color classes and
performs random edge deletion with probability $e^{-\beta}$
independently for each monochromatic edge. We denote the
resulting configuration by $E'$. The second step assigns colors
independently for each new connected component (cluster) in
$E'$, resulting in a new Potts configuration $\bss'$.

Let $E(\bss) = \{xy \in E : \sigma_x = \sigma_y\}$. Let $G$ and
$G_0$ be as in the statement of the lemma, and let $E_{1} =
E\setminus E_0$. Then
\begin{eqnarray*}
 P_G(\bss, \bss')
 & = & \sum_{E' \subseteq E(\bss)}
 \left(1- e^{-\beta}\right)^{|E'|}
e^{{-}\beta |E(\bss) \setminus E'|}
\prod_{C\in \cC(V,E')} \frac{1}{q} \prod_{x,y \in C}
\delta(\sigma'_x, \sigma'_y) \\
& \ge & e^{-\beta |E_1|} \sum_{E' \subseteq E(\bss) \setminus
E_1} \left(1- e^{-\beta}\right)^{|E'|} e^{{-}\beta
|\{E(\bss)\setminus E_1\} \setminus E'|}
\prod_{C\in \cC(V,E')}
\frac{1}{q} \prod_{x,y \in C} \delta(\sigma'_x, \sigma'_y) \\
&\ge & e^{-\beta |E_1|}  P_{G_0}(\bss, \bss'),
\end{eqnarray*}
implying that
\begin{equation}
\label{P-2-upp}
P^2_G(\bss,\bss')\geq e^{-2\beta |E_1|}
P_{G_0}^2(\bss, \bss').
\end{equation}
Next we observe that, by the definition of $H_G$, we have
\[ e^{-\beta H_{G_0}(\bss)} e^{-\beta |E_1|}
\le  e^{-\beta H_{G}(\bss)} \le   e^{-\beta H_{G_0}(\bss)}.\]
implying that
\[ \mu_{G_0}(\bss)
e^{-\beta|E_1|} \le \mu_G(\bss)
\le \mu_{G_0}(\bss)e^{\beta|E_1|}.\]
Combined with \eqref{P-2-upp} and the definition of variance
and the Dirichlet form, this proves that
\begin{equation}
\frac{\var_{\mu_G} \ g}{\E_{P_G^2}(g,g)}  \le   \frac{e^{2\beta
|E_1|} \,\var_{\mu_{G_0}} \ g}{e^{-3\beta
|E_1|}\, \E_{P^2_{G_0}}(g,g)}
\le  e^{5\beta |E_1|}\,
\frac{\var_{\mu_{G_0}} \ g}{\E_{P^2_{G_0}}(g,g)},
\nonumber
\end{equation}
and hence
\[ \ttau(P^{2}_G)
= \sup_{g}  \frac{\var_{\mu_G} \ g}{\E_{P^2_G}(g,g)} \le
 e^{5\beta |E_1|} \ttau(P^{2}_{G_0}).\]
 \qed
\end{proof}

Having established Lemma~\ref{subgraph-lem}, we are now ready
to prove Theorem~\ref{sw_ub-thm}.

\proofof{Theorem~\ref{sw_ub-thm}} Given a graph $G=(V,E)$ and a
decomposition of $V$ into two disjoint subsets $V_1$ and $V_2$,
let $E_{12}$ be the set of edges in $E$ that join $V_1$ and
$V_2$, and let $E_0=E\setminus E_{12}$. Let $G_1=G[V_1]$,
$G_2=G[V_2]$ and $G_0=(V,E_0)$.  Observing that no edge in
$G_0$ joins $V_1$ and $V_2$, we clearly have
$P_{G_0}=P_{G_1}\times P_{G_2}$ and thus
$P^2_{G_0}=P^2_{G_1}\times P^2_{G_2}$. Combining
Lemma~\ref{subgraph-lem} with the identity \eqref{prod-tau} for
$P^2_{G_1}$ and $P^2_{G_2}$, we thus have
\[
\ttau(P^{2}_{G})
\le \max\{\ttau(P^{2}_{G_1}),\ttau(P^{2}_{G_2})\} \
e^{5\beta |E_{12}|}.\]
Applying this bound recursively for the decompositions in a
hierarchical partition $\cP$ of $G$, we obtain that
\[
\ttau(P^{2}_{G})
\leq e^{5\beta\PW(G)}\prod_{x\in V}\ttau(P^2_{G[\{x\}]}).
\]
Since the inverse gap for the SW algorithms on a single site is
one, we get $\ttau(P^{2}_{G}) \leq e^{5\beta\PW(G)}$.  Combined
with the bound \eqref{mix-gap-bd}, this proves the theorem.
\qed

\section{Contour Representation}
\label{sec:Cont-Rep}

In this section, we derive a representation for the Potts model
in terms of contours and interfaces.  This representation is a
modified version of the representation of \cite{BKM}.  We use
$T=T_{L,d}$ to denote the $d$-dimensional discrete torus of
sidelength $L$, with vertex set $V=V_{L,d}=(\bbZ/L\bbZ)^d$ and
edge set $E=E_{L,d}$, and $\Om$ to denote the SW configuration
space $2^E$, i.e., the set of all edge configurations $A\subset
E$.

 We
start from the random cluster representation \eqref{FKmes}.
Given $A\subset E$, let $V(A)=\cup_{\{x,y\}\in A}\{x,y\}$, and
let $\tilde G(A)=(V(A),A)$.  Recalling the definition of $c(A)$
as the number of connected components of the graph
$G(A)=(V,A)$, we introduce $\tilde c(A)$ as the number of
 connected components of $\tilde G(A)$.
Observe that $c(A)=\tilde c(A)+|V\setminus V(A)|$ and
$$
2|A|=2d|V(A)|-|\delta_1A|-2|\delta_2A|,
$$
where
$$
\delta_kA=\{\{x,y\}\in E\setminus A ;|\{x,y\}\cap V(A)|=k\} \qquad
k=1,2.
$$
Using the notation $\delta A=\{\{x,y\}\in E\setminus A
;|\{x,y\}\cap V(A)|>0\}$ and $\|\delta
A\|=|\delta_1A|+2|\delta_2A|$, we rewrite the weight
$w(A)=Z_T\nu_T(A)$  of a configuration $A$ in \eqref{FKmes} as
\begin{equation}
\begin{aligned}
\label{FK2}
w(A)
&=\Bigl((1-p)^dq\Bigr)^{|V\setminus V(A)|}
p^{d|V(A)|}\biggl(\frac{1-p}p\biggr)^{\|\delta A\|/2}
q^{\tilde c(A)}
\\
&=
q^{\tilde c(A)}e^{-e_\dis |V\setminus V(A)|} e^{-e_\ord|V(A)|}
e^{-\kappa\|\delta A\|}
\end{aligned}
\end{equation}
where
\begin{align}
e_\dis&=-\log\Bigl((1-p)^dq\Bigr)=d\beta-\log q,
\label{e-dis-def}
\\
e_\ord&=-d\log p=-d\log\Bigl(1-e^{-\beta})=O(e^{-\beta}),
\label{e-ord-def}
\\
\kappa&=-\frac 12\log\biggl(\frac{1-p}p\biggr)
=\frac 12\log\Bigl(e^{\beta}-1\Bigr)
=\frac\beta 2 +O(e^{-\beta}).
\label{kappa-def}
\end{align}
Throughout this section we will assume that $\beta> \log
2$, so that, in particular, $\kappa>0$.

The representation \eqref{FK2}
 already shows that there are three regions of interest:
for $\beta\ll \frac 1d \log q$, the dominant configurations are
those with most vertices belonging to $V\setminus V(A)$, i.e.,
most vertices are isolated, corresponding to a ``disordered
high-temperature phase;'' for $\beta\gg \frac 1d \log q$, the
dominant configurations have most vertices in the ``ordered
region'' $V(A)$ corresponding to an ``ordered low-temperature
phase;'' and for $\beta \approx \frac 1d \log q$ and $q$ (and
hence $\kappa$) large, the dominant configurations fall into
two classes, one with mostly isolated vertices and the other
with most vertices in $V(A)$.  As we will see, even the SW
algorithm has difficulties transitioning between these two
classes, which leads to slow mixing at $\beta_0 = \frac 1d \log
q + O(q^{-1/d})$.

We will decompose $\delta A$ into several pieces called
``interfaces'' and ``contours''. More precisely, we will first
``fatten'' the set $A$ into a subset $\bV(A)$ of the continuum
torus ${\mathbf V}=({\bbR}/(L{\bbZ}))^d$ and then decompose the
boundary of $\bV(A)$ into components, some of which will be
called interfaces, and some of which will be called contours.

For an edge $e=\{x,y\}\in E$, let $\bf e$ be the set of points
in ${\bf V}$ that lie on the line between $x$ and $y$. Given
$A$, we call a closed $k$-dimensional  unit hypercube
$\bc\subset {\bf V}$ with corners in $V$ {\it occupied\/\/} if
all edges $e$ with ${\bf e}\subset \bc$ are in $A$. We then
define the set ${\bf V}(A)\subset {\bf V}$ as the
1/4-neighborhood of the union of all occupied $k$-dimensional
hypercubes, $k=1,\dots, d$, i.e., $ {\bf V}(A)= \{x\in{\bf
V}:\; \exists\, \bc \text{ occupied, such that} \,\dist(x,\bc)
\leq 1/4\}$, where $\dist(x,y)$ is the $\ell_\infty$-distance
between the two points $x$ and $y$ in the torus ${\bf V}$ and
$\dist(x,\bc)=\inf_{y\in \bc}\dist(x,y)$. Note that the set
$\bV(A)$ is a union of cubes of side-length $1/2$ with centers
in $V_{1/2} = (\frac 12\bbZ/ L\bbZ)^d$, and that the set $V(A)$
defined at the beginning of this section is just the
intersection of ${\bf V}(A)$ with the vertex set $V$ of the
discrete torus $T$.

Alternatively, one can define the set $\bV(A)$ by constructing
its complement, the ``disordered region'' $\bV\setminus\bV(A)$
as follows: Let $(E\setminus A)^\star$ be the set of $d-1$
dimensional unit hypercubes dual to the edges in $E\setminus
A$. It is then easy to see that the union of the open
$3/4$-neighborhood of $V\setminus V(A)$ and the open
$1/4$-neighborhood of $(E\setminus A)^\star$ is just the
disordered region $\bV\setminus \bV(A)$ (see
Lemma~\ref{lem:simple} in Section~\ref{sec:geom-cont}).

For $i=1,2, \ldots , d$, let $L_i$ be the fundamental loop $L_i
= \{y \in {\bf V} {: y_j=1\text{ for all }j\neq i}\}.$ If
${\bf A}$ is a union of cubes with diameter $1/2$ and centers
in $V_{1/2}$ and $\gamma$ is a component of $\partial {\bf A}$,
then the winding vector ${\bf N}(\gamma) \in \{0,1\}^d$, with
its $i$th component being equal to the number of intersections
(mod 2) of $\gamma$ with $L_i$.

\begin{definition}
\label{def:cont} Let $A$ be a configuration in $\Omega$. The
contours corresponding to the configuration $A$ are  defined as
those connected components $\gamma$ of the boundary of ${\bf
V}(A)$ which have winding number zero, $\bf N(\gamma)=0$; the
remaining connected components of the boundary of ${\bf V}(A)$
are called the interfaces corresponding to $A$; {the set of
these interfaces is called the interface network corresponding
to $A$.} We denote the set of contours corresponding to $A$ by
$\Gamma(A)$, and the {interface network} corresponding to $A$
by $\cS(A)$.

Without reference to a configuration, we say that $\gamma$ is a
contour if there exists a configuration $A$ such that
$\gamma\in\Gamma(A)$, and similarly for an interface {and an
interface network.} We define two contours  (or two interfaces,
or one interface and one contour) $\gamma$, $\gamma'$ to be
compatible, if $\dist(\gamma,\gamma')\geq 1/2$. {We define a
contour $\gamma$ and an interface network $\cS$ to be
compatible if $\gamma$ is compatible with all interfaces in
$S$.}
\end{definition}

Note that the contours and interfaces
 corresponding to a
configuration $A\in\Omega$ are orientable in the standard
topological sense (see e.g. \cite{Bott-Tu} or\cite{Alex}); in
fact, they are oriented, with an ``ordered side'' facing
$\bV(A)$, and a ``disordered side'', facing the  complement of
$\bV(A)$. Thus our contours are ``labeled contours'', with
labels $\ell\in\{\ord,\dis\}$ indicating which side is ordered,
and which is disordered.

Note also that  the contours and interfaces corresponding to a
configuration $A$ are pairwise compatible. It is not true,
however, that any set of pairwise compatible contours and
interfaces correspond to a configuration $A\in\Om$. In order to
get a one-to-one correspondence, we define the notion of
``matching labels.''

\begin{definition}
Let  $\cS$ be an {interface network}, and let $\Gamma$ be a set
of contours.  We say that $\cS\cup\Gamma$ is a  set of {\em
matching} contours and interfaces if the following conditions
hold:

\noindent (i) The contours and interfaces in $\Gamma\cup\cS$
are pairwise compatible.

\noindent (ii) The labels are  matching in the sense that, for
each component $\mathbf C$ of $\bV\setminus\bigcup_{\gamma\in
\Gamma\cup\cS}\gamma$, there exists a label $\ell({\mathbf
C})\in\{\ord,\dis\}$ such that the ordered side of each contour
(respectively, interface) faces an  ordered component (i.e., a
component with label $\ell(C)=\ord$), and similarly for the
disordered sides.

For a set of matching contours and interfaces, we denote the
union of the ordered components by $\bV_\ord$, and the union of
the disordered components by $\bV_{\dis}$.
\end{definition}

With this definition, the set of contours and interfaces
corresponding to a configuration $A\in\Om$ are clearly
matching.  It turns out (see  Corollary~\ref{cor:match} in
Section~\ref{sec:geom-cont}) that the converse is also true,
namely that any set of matching contours and interfaces
corresponds to exactly one configuration $A\in\Om$.  Using this
fact, we rewrite the partition function $Z$ as a sum over sets
of matching contours and interfaces.

To this end, we first note that the number of components
$\tilde c(A)$ is clearly equal to the number of components
$c(\bV(A))$ of the continuum set $\bV(A)$, and hence also to
the number of components of the set $\bV_\ord$. Note also that
both $\bV_\ord$ and $\bV_{\dis}$ are functions of the matching
contours and interfaces $(\Gamma,\cS)$.
 As a consequence, the weight
\eqref{FK2} can be rewritten as a function of $(\Gamma,\cS)$
according to:

\begin{equation}
\begin{aligned}
\label{FK3} w(A)
&=
q^{c(\bV(A))}
e^{-{e_\dis}|V\setminus\bV(A)|}
e^{-{e_\ord} |V\cap \bV(A)|}
e^{-\kappa \|\partial\bV(A)\|}
\\
&=
q^{c(\bV_\ord)}
e^{-{e_\dis}|\bV_{\dis}\cap V|}
e^{-{e_\ord} |\bV_\ord\cap V|}  \prod_{S\in\cS}
e^{-\kappa \|S\|} \prod_{\gamma\in\Gamma} e^{-\kappa\|\gamma\|},
\end{aligned}
\end{equation}
where $\|\partial\bV(A)\|$
 is the number of intersections of
$\partial\bV(A)$ with the continuum set $\bE = \bigcup_{e\in
E}\mathbf e$, and similarly for $\|S\|$ and $\|\gamma\|$.
Together with the already mentioned Corollary~\ref{cor:match},
about the one-to-one correspondence between configurations and
sets of matching contours and interfaces, this leads to the
desired representation of the partition function $Z=Z_T=\sum_A
w(A)$:
\begin{equation}
\label{Z-cont-rep}
Z=\sum_{\cS,\Gamma}
q^{c(\bV_\ord)}
e^{-{e_\dis}|\bV_{\dis}\cap V|}
e^{-{e_\ord} |\bV_\ord\cap V|}  \prod_{S\in\cS}
e^{-\kappa \|S\|} \prod_{\gamma\in\Gamma} e^{-\kappa\|\gamma\|},
\end{equation}
where the sum runs over sets of matching contours and
interfaces.

Next we define the interior and exterior of a contour. To this
end, we need the following geometric lemma.  Its proof is
deferred to Section~\ref{sec:geom-cont}.
\begin{lemma}
\label{lem:int-ext-exist} Let $A$ be a configuration in
$\Omega$, and let $\gamma$ be a contour of $A$.

i) The set $\bV\setminus\gamma$ has exactly two components.

ii) Let $\bold C$ and $\bold D$ be the two components of $\bold
V\setminus \gamma$, and let $S_1$, $S_2$ be two (not
necessarily compatible) interfaces that are both compatible
with $\gamma$. Then both $S_1$ and $S_2$ lie either in $\bold
C$ or $\bold D$.
\end{lemma}

{For the purpose of the next definition, it is convenient to
define the size of a set $\bW\subset\bV$ as the cardinality of
$\bW\cap V$.}

\begin{definition}
\label{def:Ext-gamma} Let $\gamma$ be a contour.  If there
exists an interface $S$ (not necessarily corresponding to the
same configuration) that is compatible with $\gamma$, we define
the exterior, $\Ext\gamma$, of $\gamma$ as the component of
$\bold V\setminus\gamma$ that contains $S$; otherwise we choose
the larger of the two components; finally, if both of these
components have the same size, we choose that containing a
distinguished point, $x_0\in\bold V$.  The interior is defined
as the
 set $\Int\gamma=\bold V\setminus (\gamma\cup\Ext\gamma)$.

Given a set of pairwise compatible contours $\Gamma$, we define
a contour $\gamma\in\Gamma$ to be an external contour in
$\Gamma$ if there exists no contour
$\gamma'\in\Gamma\setminus\{\gamma\}$ such that
$\Int\gamma\subset\Int\gamma'$. We also define the exterior of
$\Gamma$ as the set
\begin{equation}
\Ext\Gamma=\bigcap_{\gamma\in\Gamma}\Ext\gamma.
\end{equation}
Finally, we say that $\gamma_1,\dots,\gamma_n$ are mutually
external if they are pairwise compatible and
$\Int\gamma_i\cap\Int\gamma_j=\emptyset$ for all $i\neq j$.
\end{definition}

The next lemma states several properties of the set
$\Ext\Gamma(A)$ of a set of contours corresponding to a
configuration $A$ without interfaces.  Its proof is again
deferred to Section~\ref{sec:geom-cont}.

\begin{lemma}
\label{lem:extA} Let $A$ be a configuration with
$\cS(A)=\emptyset$, and let $A_\ext$ be the set of edges with
both endpoints in $\Ext\Gamma(A)$. Then the following
statements hold

\noindent (i) $\Ext\Gamma(A)$ is a connected subset of the
continuum torus $\bV$.

\noindent (ii) $\Ext\Gamma(A)$ is either a subset of the
ordered region $\bV(A)$ or the disordered region
$\bV\setminus\bV(A)$.

\noindent (iii) If $\Ext\Gamma(A)\subset \bV(A)$, then
$(V(A_\ext),A_\ext)$ is a connected subgraph of $T_{L,d}$.
\end{lemma}


\section{The Geometry of Contours and Interfaces}
\label{sec:geom-cont}

\subsection{Elementary Topological Notions}
\label{sec:elem-top}

We start by reviewing some standard notions from algebraic
topology. Let $T_{1/2}=((\half\bbZ)/L\bbZ)^d$ be the torus with
two points connected by an edge if they differ by $1/2$ (mod
$L$) in one component. Its vertex and edge sets will be denoted
by $V_{1/2}$ and $E_{1/2}$, respectively.

We define $k$-cells in $V_{1/2}$ as the $k$-dimensional
elementary cubes in $V_{1/2}$, so that the points in $V_{1/2}$
are $0$-cells, the edges are $1$-cells, etc. We also consider
the dual $V^*_{1/2}$ of $V_{1/2}$, consisting of the
barycenters of the $d$-cells in $V_{1/2}$, when considered as
subsets of the continuum torus $(\bbR/L\bbZ)^d$.  As usual,
given a $k$-cell $c$ in $V_{1/2}$, we define its dual as the
$d-k$ cell $c^*$ in $V^*_{1/2}$ that has the same barycenter as
$c$, and similarly for the dual of a $k$-cell in $V^*_{1/2}$.
Note that $(c^*)^*=c$.

Given a $k$-cell $c$, we define its boundary as the set of all
$(k-1)$-cells that are subcubes of $c$ (note that there are
$2k$ such subcubes).  More generally, for a set $K$ of
$k$-cells, we define the ($\mathbb Z_2$-)boundary of $K$ as the
set of all $(k-1)$-cells which are in the boundaries of an odd
number of cells in $K$.  We denote this boundary by $\partial
K$. The co-boundary of a $k$-cell $c$ is defined to be the set
of $(k+1)$-cells which have $c$ in their boundaries (there are
$2(d-k)$ such $(k+1)$-cells).

We often identify  a $k$-cell with the  closed $k$-dimensional
continuum cube with corners being the vertices of the discrete
cell.  In this context, $\partial K$ is identified with the
corresponding continuum boundary.

The $(d-1)$-cells in $V^*_{1/2}$ are called facets; the set of
all such facets is denoted by $F^*_{1/2}$. We say two facets
are connected (or adjacent) if they share a $(d-2)$-dimensional
cell in their boundaries.

A sequence of points $x_0,\dots,x_k\in V_{1/2}$ is called a
loop in $T_{1/2}$ if $x_0=x_k$ and $\{x_{l},x_{{l}+1}\}\in E_{1/2}$ for all ${l}\in\{0,k-1\}$.  Such a
loop is called a fundamental loop in the $i^{\text{th}}$
lattice direction if $k=2L$ and all edges point in the
$i^{\text{th}}$ lattice direction, and it is called an
elementary loop if $k=4$ and neither $x_0=x_2$ nor $x_1=x_3$.
Note that there are exactly $2d(2d-2)(2 L)^d$ elementary loops
in $T_{1/2}$.

Consider now a set of edges $X\subset E_{1/2}$ and its dual
$X^* = \{e^* : e \in X\}$, and assume that ${X^*}$ is
orientable. If $\partial X^*=\emptyset$, then we say that $X^*$
is an orientable closed surface, and we define the
$\bbZ_2$-winding vector of $X^*$ as the vector ${\bf
N}(X^*)=(N_1,\dots,N_d) \in \{0,1\}^d$ with $N_i$ equal to the
number of times $X^*$ intersects a fundamental loop in the
$i^{\text{th}}$ lattice direction $\mod 2$.

%
%

\subsection{Preliminaries}

Our first lemma summarizes several simple properties of the
construction used in the definition of contours. It involves
both facets in $V^*_{1/2}$, the objects dual to the $1$-cells
in $V_{1/2}$, and $(d-1)$-dimensional unit hypercubes dual to
the edges in $E$.  While the first will be considered to be
abstract objects in the sense of algebraic topology, the second
will be considered to be closed hypercubes in $\bV$. We trust
that this does not cause any confusion to the reader.

{We need some notation. Given a set $D\subset E$, let
$D^\star$ be the set of $d-1$ dimensional unit hypercubes dual
to the edges in $D$, let $V_-(D)$ be the set of all vertices
$x\in V$ such that  all edges $\{x,y\}\in E$ containing the
vertex $x$ lie in $D$ (alternatively, this set can be defined
as $V_-(D)=V\setminus V(E\setminus D)$), and let
$\bV_{\dis}(D)$ be the union of the open $3/4$-neighborhood of
$V_-(D)$ and the open $1/4$-neighborhood of $D^\star$. Note
that the set $V_-(D)$ and hence the set $\bV_{\dis}(D)$ depends
implicitly on $E$. Since $E$ is fixed throughout, we suppress
this dependence in our notation.}

\begin{lemma}
\label{lem:simple}

i) For $A\in\Om$, the boundary $\partial \bV(A)$ of ${\bf
V}(A)$ is regular in the sense that each $(d-2)$-cell with
corners in $V^*_{1/2}$ is shared by either zero or two facets
in $\partial \bV(A)$.

ii) Let $A\in\Om$, and let $\bold C$ be a component of
$\bV\setminus\partial \bV(A)$.  Then $\bold C$ is either a
component of $\bV(A)$ or of $\bV\setminus \bV(A)$.

iii) Let $\bold C_1,\dots,\bold C_k$ be the connected
components of $\bV(A)$, let $V_i=V\cap\bold C_i$, and let $A_i$
be the set of edges whose endpoints lie in ${V}_i$.  Then
$(V_1,A_1)$, $\dots$, $(V_k,A_k)$ are the connected components
of $(V,A)$, and $\mathbf C_i={\bV(A_i)}$.

iv) {Let $D=E\setminus A$ and let $\bV_{\dis}(D)$ be as
defined above.} Then $\bV\setminus \bV(A)=\bV_{\dis}(D)$.

v) Let $\mathbf C$ be a component of $\bV_{\dis}({D})$,
and let ${D_{\mathbf C}}$ be the set of edges in $E$ whose
midpoint lies in $\mathbf C$. Then $\mathbf C=\bV_{\dis}({D_{\mathbf C}})$.

\end{lemma}

\begin{proof}
(i) Given a configuration $A\in \Om$, let $V_{1/2}(A)$ be the
intersection of ${\bf V}(A)$ with the vertex set $V_{1/2}$ of
the discrete torus $T_{1/2}$. The boundary of ${\bf V}(A)$ can
then be rewritten as the union of all facets that are dual to
an edge $e \in E_{1/2}$ joining $V_{1/2}(A)$ to its complement
in $V_{1/2}$.  Using this fact, we easily prove the first
statement of the lemma.

Indeed, let $e$ be a $(d-2)$-cell with corners in $V^*_{1/2}$,
and let $f_1,f_2,f_3,f_4$ be the four facets in the co-boundary
of $e$. Then $f_1,f_2,f_3,f_4$ are dual to edges $x_1x_2,
x_2x_3, x_3x_4, x_4x_1$, where the points $x_1, x_2, x_3, x_4$
comprise a closed path of length four in the torus $T_{1/2}$.
Let $c(x_1),\dots,c(x_4)$ be the four $d$-cells with centers
$x_1,\dots x_4$. Exactly one of these four cubes, say the cube
$c(x_1)$, will be a cube whose center lies in the original
vertex set $V$ (recall that vertices in $V$ have twice the
spacing of those in $V_{1/2}$). Our convention of filling in
only those hypercubes in the original torus $T_{L,d}$ whose
edges are all in $A$ implies that either none or all or exactly
one of the four cubes $c(x_1),\dots,c(x_4)$ lies in $\bV(A)$.
In the first two cases, none of the facets $f_1,f_2,f_3,f_4$
are in the boundary of $\bV(A)$, and in the third case exactly
two are in the boundary of $\bV(A)$, which proves that each
$(d-2)$-cell
 with corners
in $V^*_{1/2}$ is shared by either zero or two facets, as
claimed.

(ii) This is obvious.

(iii) This now  follows immediately from our fattening
procedure for the ordered region, which  respects the component
structure of $(V,A)$.

{(iv) We first prove that $\bV_{\dis}(E\setminus A)\subset
\bV\setminus \bV(A)$.  Consider first an edge $e\in E\setminus
A$ and its dual $e^\star$.  We claim that all points with
distance less than $1/4$ from $e^\star$ lie in $\bV\setminus
\bV(A)$. Since the set $\bV(A)$ increases if $A$ increases, it
is clearly enough to prove this statement for
$A=E\setminus\{e\}$, in which case it follows immediately from
the way we set up our fattening procedure for $\bV(A)$.  In a
similar way, one proves that all points with distance less than
$3/4$ from the vertices in ${V_-(D)}=V\setminus V(A)$ lie
in $\bV\setminus \bV(A)$.

To  prove that $\bV\setminus \bV(A)\subset\bV_{\dis}(E\setminus
A)$, let $x\in \bV\setminus \bV(A)$, and let $\bc\subset\bV$ be
a $d$-dimensional unit cube with corners in $V$ such that
$x\in\bc$.  If $x$ has distance less than $1/4$ from the center
of $\bc$, then at least one edge $e\subset \bc\cap V$ must lie
in $E\setminus A$, since otherwise $\bc$ would have been filled
in our fattening procedure, contradicting $x\in \bV\setminus
\bV(A)$.  But the $1/4$ neighborhood $\mathcal
N_{1/4}(e^\star)$ of $e^\star$ contains all points with
distance less than $1/4$ from the center of $\bc$, implying
that $x\in \mathcal N_{1/4}(e^\star)\subset
\bV_{\dis}(E\setminus A)$. If $x$ has distance $1/4$ or more
from the center of $\bc$, it must have distance less than $1/4$
from a $d-1$ dimensional unit cube $\bc_1$ in the boundary of
$\bc$.  If the projection, $x_1$, of $x$ onto $\bc_1$ has
distance less than $1/4$ from the center of $\bc_1$, then one
of the edges in $\bc_1\cap V$ must lie in $E\setminus A$, which
again implies $x\in \mathcal N_{1/4}(e^\star)\subset
\bV_{\dis}(E\setminus A)$. Continuing inductively, we are left
with the case that $x$ has distance at most $1/4$ from one of
the corners, $y$, of $\bc$.  But this means that none of the
edges containing $y$ can lie in $A$, thus $y\in V\setminus
V(A)$, and hence $x\in \mathcal N_{3/4}(V\setminus V(A))\subset
\bV_{\dis}(E\setminus A)$.

(v) Noting that the midpoint of an edge lies in
$\bV_{\dis}(E\setminus A)$ if and only if its dual lies in
$\bV_{\dis}(E\setminus A)$, this is an immediate consequence of
our fattening procedure for the disordered region.} \qed
\end{proof}

\begin{lemma}
\label{lem:NS1=NS2} Let $S_1$ and $S_2$ be two interfaces with
$S_1\cap S_2=\emptyset$.  Then $\bold N(S_1)=\bold N(S_2)$.
\end{lemma}

\begin{proof}
The interfaces $S_1$ and $S_2$ are closed orientable
submanifolds of the torus $(\bbR/L\bbZ)^d$.  In the language of
algebraic topology, the winding numbers ${\bf N}(S_1)$ and
${\bf N}(S_2)$ are the Poincare duals of the submanifolds $S_1$
and $S_2$.  The Poincare dual of the transverse intersection of
two such submanifolds is then given by the wedge product of the
Poincare duals of the submanifolds, see \cite{Bott-Tu}, Section
6.  Since empty intersection is a special case of transverse
intersection, we conclude that the wedge product of ${\bf
N}(S_1)$ and ${\bf N}(S_2)$ must be zero. Let $\vec e_i$ be the
unit vector whose $j^{\text{th}}$ coordinate is $\delta_{i,j}$.
Recalling that $\vec e_i\wedge\vec e_i=0$ and $\vec
e_i\wedge\vec e_j=-\vec e_j\wedge\vec e_i$ for all $i,j$, the
condition ${\bf N}(S_1)\wedge{\bf N}(S_2)=0$ is equivalent to
the $d\choose 2$ conditions
$N_i(S_1)N_j(S_2)-N_j(S_1)N_i(S_2)=0$, which implies that ${\bf
N}(S_1)$ and ${\bf N}(S_2)$ are multiples of each other. Since
both are different from 0, we conclude that ${\bf N}(S_1)={\bf
N}(S_2)$. \qed
\end{proof}

\smallskip

We close this section with an elementary lemma about
``cutsets''. It is best formulated in the context of a general
connected graph $G=(V,E)$.  As usual, a subset $E'\subset E$ is
a {\em cutset} if $(V,E\setminus E')$ is disconnected.  It is
called
 a {\em minimal} cutset if no proper subset of $E'$
is a cutset.  The following lemma is elementary; its proof is
left to the reader.

\begin{lemma}\label{lem:cutset-conn}
Let $G=(V,E)$ be a connected graph, and let $W\subset V$.  If
the edge-boundary of $W$ is a minimal cutset then both the
induced graph on $W$ and $V\setminus W$ are connected.
\end{lemma}

\subsection{Proofs of Lemma~\ref{lem:int-ext-exist} and
\ref{lem:extA}}

We start with the proof of Lemma~\ref{lem:int-ext-exist}.

\smallskip

\proofof{Lemma~\ref{lem:int-ext-exist}} i) $\gamma$ has no
boundary, is orientable, and has winding number zero. Therefore
any closed path intersects $\gamma$ an even number of times,
implying that $\gamma$ is the boundary of some open set $\bold
C$.  Let $\bold D$ be the set $\bold D=\bold V\setminus (\bold
C\cup\gamma)$.  Then  both $\bold C$ and $\bold D$ must be
connected. Indeed, considering $\gamma$ as dual to a minimal
cutset on the half-integer lattice, the connectedness of $\bold
C$ and $\bold D$ follows immediately from the corresponding
statement (Lemma~\ref{lem:cutset-conn}) for minimal cutsets.

ii) The interfaces $S_1,S_2$ are connected subsets of $\bold
V$. Since they do not intersect $\gamma$, each of them must lie
in one of the connected components of $\bold V\setminus\gamma$.
We will have to prove that they both lie in the same component
of $\bold V\setminus \gamma$.

Assume, by contradiction, that $S_1\subset \bold C$ and
$S_2\subset \bold D$. Since $\bold C\cap\bold D=\emptyset$,
this implies $S_1\cap S_2=\emptyset$, which in turn, by
Lemma~\ref{lem:NS1=NS2}
 implies that ${\bf
N}(S_1)={\bf N}(S_2)$.  Together with the fact that ${\bf
N}(\gamma)=0$ we conclude that the winding number of
$\gamma\cup S_1\cup S_2$ is zero, implying that any loop in
$\bold V$  must intersect $\gamma\cup S_1\cup S_2$ an even
number of times.

Consider now a component $N_i(S_1)$ of the winding vector ${\bf
N}(S_1)$ that is not equal to $0$, and let $\omega$ be a
fundamental loop in the $i^{\text{th}}$ direction, oriented in
an arbitrary but fixed fashion.  To reach our contradiction, we
will modify $\omega$ in such a way that it does not intersect
$\gamma$ or $S_2$, while intersecting $S_1$ an odd number of
times.  First, we note that the original loop $\omega$
intersects $S_1$ an odd number of times.  If it does not
intersect $\gamma$, then it does not intersect $S_2$ either,
since $S_2$ lies in $\bold D$, while $S_1$ lies in $\bold C$.
If $\omega$ intersects $\gamma$, let $x$ be one of the
intersection points, and let $y$ be the next intersection point
(since ${\bf N}(\gamma)=0$, the number of these intersection
points must be even, so there must be such a $y$). Recalling
that $\gamma$ is connected, we now replace the segment of
$\omega$ that joins $x$ to $y$ by a path in $\gamma$, and then
deform this segment in such a way that it lies completely in
$\bold C$ without intersecting $S_1$ (this is possible since
$\dist(S_1,\gamma)\geq 1/2$).  Given that the original segment
from $x$ to $y$ did not intersect $S_1$ since $S_1\subset \bold
C$ while the segment was a path in $\bold D$, we have not
changed the number of intersections with $S_1$, so this number
is still odd. Repeating this step until we have no
intersections with $\gamma$, we remove all parts of $\omega$
that lie outside of $\bold C$, ending up with a path inside
$\bold C$ that intersects $S_1$ an odd number of times, as
desired.  Note that the final $\omega$ may not be a lattice
path, but nevertheless it gives the desired contradiction,
since winding numbers are topological invariants in the
continuum as well. \qed

\begin{definition}
Let $\gamma$ and $\gamma'$ be two contours.  We say that
$\gamma$ and $\gamma'$ are mutually external (and write
$\gamma\perp\gamma'$) if $\gamma$ is compatible with $\gamma'$
and $\Int\gamma\cap\Int\gamma'=\emptyset$, and we say that
$\gamma$ lies inside of $\gamma'$ (and write $\gamma<\gamma'$)
if $\gamma$ is compatible with $\gamma'$ and $\Int
\gamma\subset\Int\gamma'$.
\end{definition}

\begin{lemma}
\label{lem:partial-order} Let $\gamma$, $\gamma'$ and
$\gamma''$ be contours.

\noindent (i) If $\gamma$ and $\gamma'$ are compatible, then
exactly one of the following  three holds: $\gamma<\gamma'$,
$\gamma'<\gamma$, or $\gamma\perp\gamma'$.

\noindent (ii) If $\gamma\perp\gamma'$, then
$\dist(\Int\gamma,\Int\gamma')\geq 1/2$, and if
$\gamma<\gamma'$, then $\dist(\Int\gamma,\Ext\gamma')\geq 1/2$.

\noindent (iii) If $\gamma<\gamma'$ and $\gamma'<\gamma''$ then
$\gamma<\gamma''$.

\noindent (iv)If $\gamma<\gamma'$ and $\gamma'\perp\gamma''$
then $\gamma\perp\gamma''$.
\end{lemma}

\begin{proof}
(i) Let $\Int\gamma$ and $\Ext\gamma$ be the two components of
$\bold V\setminus \gamma$.  Since $\gamma'$ is connected and
$\gamma\cap \gamma'=\emptyset$, $\gamma'$ must lie in one of
the two components of $\bV\setminus \gamma$; therefore we have
that either $\gamma'\subset \Int\gamma$ or
$\gamma'\subset\Ext\gamma$.  Since the same statement holds
with the roles of $\gamma$ and $\gamma'$ exchanged, we get that
exactly one of the following four cases must hold:
\begin{align}
\label{p-ord-case1}
\gamma\subset\Ext\gamma'
\quad&\text{and}\quad
\gamma'\subset\Ext\gamma
\\
\label{p-ord-case2}
\gamma\subset\Int\gamma'
\quad&\text{and}\quad
\gamma'\subset\Ext\gamma
\\
\label{p-ord-case3}
\gamma\subset\Ext\gamma'
\quad&\text{and}\quad
\gamma'\subset\Int\gamma
\\
\label{p-ord-case4}
\gamma\subset\Int\gamma'
\quad&\text{and}\quad
\gamma'\subset\Int\gamma
\end{align}
We claim that the last case is impossible.

To prove this, we first show that \eqref{p-ord-case4} implies
that $\Ext\gamma\cap\Ext\gamma'\neq\emptyset$.  Recalling
Definition~\ref{def:Ext-gamma} of $\Ext\gamma$, note that
$\Ext\gamma$ is defined differently in several distinct cases.
Let us first consider the case that there exists an interface
$S$ with $\dist(S,\gamma)\geq 1/2$ and $S\subset\Ext\gamma$. If
$\dist(S,\gamma')<1/2$, then also $\dist(S,\Int\gamma)<1/2$ by
our assumption that $\gamma'\subset\Int\gamma$.  But this is
not compatible with $\dist(S,\gamma)\geq 1/2$ and
$S\subset\Ext\gamma$. Thus we have $\dist(S,\gamma')\geq 1/2$.
But if $\dist(S,\gamma')\geq 1/2$, then $S\subset\Ext\gamma'$
as well, implying in particular that
$\Ext\gamma\cap\Ext\gamma'\neq\emptyset$.

Let us now consider the cases in the definitions of
$\Ext\gamma$ and $\Ext\gamma'$ such that there is no interface
$S$ compatible with either $\gamma$ or $\gamma'$. Consider the
subcase that $\Ext\gamma$ is defined by size, so that
$|\Ext\gamma \cap V|>|\Int\gamma \cap V|$ and $|\Ext\gamma'
\cap V|\geq|\Int\gamma' \cap V|$. This implies that
$|\Ext\gamma \cap V|> |V|/2$ and $|\Ext\gamma' \cap V|\geq
|V|/2$ which in turn implies that
$\Ext\gamma\cap\Ext\gamma'\neq\emptyset$. The case where
$\Ext\gamma'$ is defined by size is strictly analogous, so we
are left with the subcase where both $\Ext\gamma$ and
$\Ext\gamma'$ are defined by containing the distinguished point
$x_0$.  Again, this implies that
$\Ext\gamma\cap\Ext\gamma'\neq\emptyset$.

Now we show that the condition
$\Ext\gamma\cap\Ext\gamma'\neq\emptyset$ rules out the case
\eqref{p-ord-case4}.  Let $u\in \Ext\gamma\cap\Ext\gamma'$, and
let $x\in\Int\gamma$.  Consider a path $\omega$ from $x$ to
$u$, and let $y$ be the last exist point from $\Int\gamma$
along $\omega$.  This implies $y\in\gamma$, and by our
assumption \eqref{p-ord-case4}, we therefore have $y\in
\Int\gamma'$.   But this implies there exists a point
$z\in\gamma'\subset\Int\gamma$ after $y$, which is a
contradiction.  Thus case \eqref{p-ord-case4} is impossible.

Next we prove that the three remaining cases imply that
$\Int\gamma'\cap\Int\gamma=\emptyset$,
$\Ext\gamma'\cap\Int\gamma=\emptyset$ and
$\Int\gamma'\cap\Ext\gamma=\emptyset$, respectively, showing
that either $\gamma\perp\gamma'$, $\gamma<\gamma'$ or
$\gamma'<\gamma$, respectively.

We prove all three statements in one sweep, by setting
$A=\Ext\gamma$ and $A'=\Ext\gamma'$ in the first case,
$A=\Ext\gamma$ and $A'=\Int\gamma'$ in the second case, and
$A=\Int\gamma$ and $A'=\Ext\gamma'$ in the third case.  Our
assumption then reads $\gamma\subset A'$ and $\gamma'\subset
A$, and our claim is $B\cap B'=\emptyset$, where
$B=A^c\setminus\gamma$ and $B'=(A')^c\setminus\gamma'$. In a
preliminary step, we  prove that $\gamma\subset A'$ implies
$B\setminus B'\neq\emptyset$.  Indeed, assume the contrary,
i.e., $B\subset B'$.  Taking the closure on  both sides, this
gives $\gamma\cup B\subset \gamma'\cup B'$, and hence
$\gamma\subset \gamma'\cup B'=(A')^c$, a contradiction.    Now
we prove the main claim $B\cap B'=\emptyset$.  Assume the
contrary, that there exists an $x\in B\cap B'$. From our
preliminary claim, we also know that there exists a $y\in
B\setminus B'$.  Since $B$ is connected, we conclude that there
must be a path $\omega\subset B$ from $x$ to $y$.  Let $z$ be
the first time this path exits $B'$. Then $z\in B\cap\partial
B'=B\cap\gamma'$. Thus $B\cap\gamma'\neq\emptyset$, a
contradiction.

(ii) Both statements follow from the observation that if $A\cap
A' = \emptyset$, then $\dist(\partial A,
\partial A') \leq \dist (A,A')$.

(iii)  Assume that $\gamma < \gamma'$ and $\gamma' < \gamma''$.
Then $\Int\gamma \subset \Int\gamma' \subset \Int\gamma''$, so
we need only prove that $\gamma$ and $\gamma''$ are compatible,
i.e. that $\dist(\gamma,\gamma') \geq 1/2$. On the other hand,
$\Int\gamma\subset\Int\gamma'$.  Taking the closure of both
sides and using a trivial inclusion, this implies that
$\gamma\subset\gamma \cup \Int\gamma
\subset\gamma'\cup\Int\gamma'$.  Thus
$$
\dist(\gamma,\gamma'') \geq
\dist(\gamma'\cup\Int\gamma',\gamma''\cup\Ext\gamma'') =
\dist(\Int\gamma',\Ext\gamma'')\geq 1/2
$$
where we used (ii) in the last step.

(iv) This is proved strictly analogously to the proof of (iii).
\qed
\end{proof}

The next lemma is an easy corollary of
Lemma~\ref{lem:partial-order}.

\begin{lemma}
\label{lem:ext-connect} Let $\Gamma$ be a set of pairwise
compatible contours. Then $\Ext\Gamma$ is a connected subset of
$\bV$.
\end{lemma}

\begin{proof}
Let $\Gamma_\ext$ be the set of external contours in $\Gamma$.
It follows immediately from the definition of $\Ext\Gamma$ and
the last lemma that $\Ext\Gamma=\Ext\Gamma_\ext$. It is
therefore enough to consider a set $\Gamma$ of mutually
external contours.  We prove the statement by induction on the
number of contours in $\Gamma$.  The statement is trivial if
$\Gamma = \emptyset$.  Assume the statement is proved for
$\Gamma=\{\gamma_1,\dots,\gamma_{n-1}\}$.  Adding an additional
mutually external contour $\gamma_n$ will not change the
connectivity. Indeed, let $x,y\in\Ext(\Gamma\cup\{\gamma_n\})$,
and let $\omega$ be a path in $\Ext\Gamma$ that joins $x$ to
$y$. If $\omega$ does not intersect $\gamma_n\cup\Int\gamma_n$
there is nothing to prove.  Otherwise, let $x'\in\gamma_n$ be
the first entry point into $\gamma_n\cup\Int\gamma_n$, and let
$y'\in\gamma_n$ be the last exit point from
$\gamma_n\cup\Int\gamma_n$.  Since $\gamma_n$ is connected, we
can replace  the path $\omega$ from $x'$ to $y'$ by a path in
$\gamma_n$, leading to a path $\omega'$ joining $x$ and $y$ in
$\gamma_n\cup\Ext(\Gamma\cup\{\gamma_n\})$. By deforming
$\omega'$ this immediately leads to a path in
$\Ext(\Gamma\cup\{\gamma_n\})$, proving the lemma. \qed
\end{proof}

\smallskip

We are now ready to prove Lemma~\ref{lem:extA}.

\proofof{Lemma~\ref{lem:extA}}
(i) Since the contours corresponding to a configuration $A$ are
pairwise compatible, this statement follows immediately from
the previous lemma.

(ii) Using the fact that the boundary of $\bV(A)$ is equal to
the union over all contours in $\Gamma(A)$, we conclude that
$\Ext\Gamma(A)$ is  a connected subset of $\bV\setminus\partial
\bV(A)$. But this implies that $\Ext\Gamma(A)\subset \bV(A)$ or
$\Ext\Gamma(A)\subset \bV\setminus\bV(A)$, as claimed.

(iii) This follows immediately from (i) and the third statement
of Lemma~\ref{lem:simple}. \qed

\subsection{Isoperimetric Estimates}

We need the notion of  {\em diameter}, $\diam\gamma$, of a
contour $\gamma$. {To this end, we consider sets
${\mathbf S}_{k}^{(i)}$ of the form
$$
{\mathbf S}_{k}^{(i)}=\{x\in \bV\colon x_i=k\},
$$
and define $I_i=I_i(\gamma)$ as
$$
I_i=\{k\in\bbZ/L \bbZ\colon {\mathbf S}_{k}^{(i)}\cap\gamma\neq\emptyset\}.
$$
By the fact that $\gamma$ is a connected subset of $\bV$, we
have that $I_i$ is a set of consecutive integers mod $L$ (i.e.,
it is a connected subset of $\bbZ/L \bbZ$).
%
%
We define the diameter, $\diam_i\gamma$ of $\gamma$ in the
direction $i$ as the number of points in $I_i$, and the
diameter of $\gamma$ as
$\diam\gamma=\max_{i=1,\dots,d}\diam_i\gamma$.}

\begin{lemma}\label{lem:iso}
For every contour $\gamma$, we have
\begin{equation}
\label{g-larger-diam-g}
\|\gamma\|\geq{2\,}\diam\gamma\,
\end{equation}
and
\begin{equation}
\label{iso1}
|\Int\gamma\cap V|\leq \frac 12\|\gamma\| \, \diam \gamma\,.
\end{equation}
\end{lemma}

\begin{proof}
{
Let $S_k^{(i)}={\mathbf S}_{k}^{(i)}\cap V$, and let
$E_k^{(i)}$ be the set of edges $xy\in E$ such that both $x$
and $y$ lie in $S_k^{(i)}$. Consider the configuration $A$
which has $\gamma$ as its only contour.  If $A\cap
E_k^{(i)}=E_k^{(i)}$, then ${\mathbf S}_{k}^{(i)}\subset
\bV(A)$ (this follows immediately from the fattening procedure
used to define $\bV(A)$, and if $A\cap E_k^{(i)}=\emptyset$,
then ${\mathbf S}_{k}^{(i)}\subset \bV\setminus\bV(A)$ (this
follows  from Lemma~\ref{lem:simple}).  In either case,
$\gamma$ must have an empty intersection with ${\mathbf
S}_{k}^{(i)}$. (In fact, $\gamma$ must have distance at least
$1/4$ from this set). For $k\in I_i$, the set $E_k^{(i)}$
therefore must contain at least one edge in $\delta A$,
implying that $\gamma$ has at least one intersection with the
edges in $E_k^{(i)}$.  In fact, by a simple parity argument,
there must be at least two such intersections. This immediately
implies that $\|\gamma\|\geq 2\, \diam_i(\gamma)$ for all
$i=1,\dots,d$, which proves the bound \eqref{g-larger-diam-g}.

Consider now the sets $V_i=\{x\in V\mid x_i\in I_i\}$ and
$$
\bV_i=\{x\in \bV\colon \exists k\in I_i
\text{ s.t. } |x_i-k|\leq 3/4\}.
$$
Since $I_i$ consists of consecutive integers mod $L$, the set
$V_i$ is a connected subset of the discrete torus $V$, and
$\bV_i$ is a connected subset of $\bV$. If
$\diam_i\gamma>0$, the set $\bV_i$ is non-empty, and
$\gamma\subset\bV_i$.  If $\diam_i\gamma=0$, all the edges
intersecting $\gamma$ must be in the direction $i$, and since
$\gamma$ is connected, they must all lie in one plane of the
torus.  In other words, there must be a set of the form $\{x\in
\bV\colon |x_i-k+1/2|\leq 1/4\}$ such that $\gamma$ is
contained in this set.  We denote it again by $\bV_i$\,.

While it is in general not true that $\Int\gamma\subset
\bV_i$\,, we claim that this is true if there exists an
interface $S$ that is compatible with $\gamma$.  Indeed, let
$S$ be such an interface, and let $S_1$ be one of the sets
${\mathbf S}_{k}^{(i)}\subset\bV\setminus \bV_i$ (if there is
no such set, $\bV_i=\bV$ and there is nothing to prove).
Repeating the proof of Lemma~\ref{lem:int-ext-exist}, we see
that $S_1$ and $S$ must lie in the same component of
$\bV\setminus \gamma$, which by Definition~\ref{def:Ext-gamma},
must be the exterior of $\gamma$.  This proves that
$\Ext\gamma\cap(\bV\setminus\bV_i)\neq\emptyset$. Taking into
account the connectedness of $\gamma$, this proves that
$\Int\gamma\subset\bV_i$.

For contours whose exterior is defined by the existence of a
compatible interface $S$, this immediately implies the bound
\eqref{iso1}.  Indeed, assume that there exists an interface
$S$ such that $S\subset \Ext\gamma$. Without loss of
generality, let us assume that the first component of the
winding vector ${\mathbf N}(S)$ is 1. This implies that every
line in the 1-direction intersects $S$ at least once. Hence,
for all $x \in W=\Int \gamma \cap V$, a line through $x$ in the
1-direction intersects $\gamma$ at least twice.  Since
$W\subset \bV_1$, this shows that
\[|\Int\gamma \cap V| \le \frac{1}{2}\|\gamma\|\, \diam_1\gamma\,,
\]
which implies \eqref{iso1}.

We are thus left with contours $\gamma$ for which
$|W|=|\Int\gamma\cap V|\leq L^2/2$. For these contours, the}
isoperimetric inequality of Bollob\'as-Leader \cite{iso}
implies that
\[ \|\gamma\|\geq |\partial_{\rm edge} W| \ge \min_{i=1,\ldots, d} 2i|W|^{1-1/i} L^{d/i - 1}
\ge 2\min_i |W|^{1-1/i} L^{d/i - 1} = 2|W|^{1-1/d}\,,\]
where we used the notation $\partial_{\rm edge} W$ for the
edge-boundary of $W$. {To complete the proof of
\eqref{iso1}, we will want to show that
\begin{equation}
\label{Vleqdiamd}
|W|=|\Int\gamma\cap V|\leq (\diam\gamma)^d\,.
\end{equation}
This bound is trivial if $\diam\gamma=L$, so let us assume that
$\diam\gamma<L$.  Let $\mathbf C=\bigcap_{i=1,\dots,d}\bV_i$.
Since $\gamma\subset\mathbf C$ and $\bV\setminus \mathbf C$ is
connected, we must have that either $\Int\gamma\subset\mathbf
C$ or $\Ext\gamma\subset\mathbf C$.  In the first case, the
bound \eqref{Vleqdiamd} is again trivial, and in the second
case we use that
\[|\Int\gamma\cap V|\leq |\Ext\gamma\cap V|
\leq |\mathbf C\cap V|\leq (\diam\gamma)^d\,,
\]
in completing the proof.}
\end{proof}

{In addition to the above lemma, we will also need bounds
on the number of contours and interfaces of a given size. More
precisely, we need the following lemma.

\begin{lemma}
\label{lem:cont-estimates}
There exists an absolute constant $C<\infty$, such that the following statements hold 

i) Let $\gamma_0$ be a contour or an interface, and let $k\geq
2$.  Then the number of contours $\gamma$ such that
$\|\gamma\|=k$ and $\dist(\gamma,\gamma_0)<1/2$ is at most
$\|\gamma_0\| (Cd)^k$.

ii) Fix $k\geq L^{d-1}$.  Then the number of interfaces $S$ in
$\bV$ such that $\|S\|=k$ is at most $L(Cd)^k$.

iii) Fix a vertex $x\in V$.  Then the number of contours
$\gamma$ such that $\|\gamma\|=k$ and $x\in \Int\gamma$ is at
most $ (Cd)^k$.
\end{lemma}

\begin{proof}
i) Clearly, there is a 1-1 correspondence between contours
$\gamma$ and the set of edges $E_\gamma$ intersecting them
(taking an edge twice if $\gamma$ has two intersections with
it, so that $\|\gamma\|$ is equal to the number of edges in
$E_\gamma$), with an analogous statement holding for
interfaces. Defining a suitable neighborhood relation on the
edges in $E_{L,d}$, the set $E_\gamma$ is a connected subset in
a graph of maximal degree bounded by $Kd$ for some $K<\infty$,
and $E_\gamma\cup E_{\gamma_0}$ is connected if
$\dist(\gamma,\gamma_0)<1/2$. Using standard results on the
number of connected sets in a graph of given maximal degree, we
immediately obtain statement i).

ii)  Let $x_0$ be an arbitrary point in $V_{L,d}$, let $L_i$ be
the straight line through the $x_0$ in  $i$ direction, and let
$E_i$ be the set of edges whose line-segments lie in $L_i$.  If
$S$ has non-zero winding number in the direction $i$, then $S$
must intersect each line in direction $i$ at least once,
implying that $E_S$ must contain at least one edge in $E_i$.
Taking $S_0$ to be the union $E_1\cup\dots\cup E_d$, we see
that an arbitrary interface $S$ must contain at least one edge
in $S_0$.  Since $S_0$ contains $dL$ edges, the result ii) now
follows by the argument  used in the proof of i).

iii) Let $x=(x_1,\dots,x_d)$, and for $r=1,2,...,L$, let $W_r$
be the cube of all points $y\in V$ such that $-\frac r2<
y_i-x_i \leq \frac r2$ for all $i=1,\dots, d$. Choose $R$ to be
the largest $r$ such that $W_r\subset \Int\gamma$.  Then
$\gamma$ must intersect one of the $2dR^{d-1}$ edges joining
$W_{R}$ to $W_{R+1}$, implying that the number of contours in
question (corresponding to a fixed $R$) is at most
\[
2dR^{d-1} (Cd)^{k-1}
\leq 2d |W_{R}|  (Cd)^{k-1}\leq \frac d2 k^2 (Cd)^{k-1},
\]
where in the last step we used the previous lemma to bound
$|W_{R}|\leq \frac 14 \|\gamma\|^2=\frac{k^2}4$. To complete
the proof, we will have to sum over $R$.  But $R^d=|W_R|\leq
\frac{k^2}4$ implies that $R\leq k$; thus the summation over
$R$ gives another factor of $k$, leading to the bound
\[\frac d2 k^3 (Cd)^{k-1}
\leq (8Cd)^k
\]
for the number of contours $\gamma$ such that $\|\gamma\|=k$
and $x\in \Int\gamma$.  This proves iii).

\end{proof}
}

\subsection{Matching Contours and Interfaces}

In this section we will show that the partition function $Z$
can be written as a sum over sets of matching contours and
interfaces.  To this end, we establish a sequence of lemmas. We
start with the following lemma.

\begin{lemma}
\label{lem:cont-base-case} Let $\gamma$ be a contour.  Then
there exists a configuration $A$ such that $\cS(A)=\emptyset$
and $\Gamma(A)=\{\gamma\}$.
\end{lemma}

\begin{proof}
By definition, there exists a configuration $A_1$ such that
$\gamma$ is one of the contours corresponding to $A_1$.  Thus
$\gamma$ is a connected component of $\partial \bV(A_1)$, and
hence a connected component of the boundary of one of the
components, $\mathbf C$, of $\bV(A_1)$.  Let $A_2$ be the set
of edges in $A_1$ whose endpoints both lie in $\mathbf C$. Then
$(V(A_2),A_2)$ is a component of $(V,A_1)$, and, by
Lemma~\ref{lem:simple} (iii), the set $\bV(A_2)$ is nothing but
the component $\mathbf C$ of $\bV(A_1)$.

Thus $\gamma$ is a component of $\partial\bV(A_2)$.  Consider
now the complement $\mathbf D=\bV\setminus \bV(A_2)$, and its
components $\mathbf D_0,\dots\mathbf D_k$. Then $\gamma$ is the
boundary of one of these components, say $\mathbf D_0$.  Let
$A_0$ be the set of edges whose midpoint lies in $\bV\setminus
\mathbf D_0$. By Lemma~\ref{lem:simple} (iv) and (v), we have
that $\mathbf D_0=\bV\setminus \bV(A_0)$, which in turn implies
that $\gamma=\partial\bV(A_0)$. \qed
\end{proof}

\begin{lemma}
\label{lem:int-base-case} Let $\cS$ be an interface network.
Then there exists a configuration $A$ such that $\cS(A)=\cS$
and $\Gamma(A)=\emptyset$.
\end{lemma}

\begin{proof}
The proof is identical to that of the previous lemma.
\end{proof}

Recall that each contour has an ordered and a disordered side.
We call $\gamma$ a contour with  external label {\em ord} if
the side facing $\Ext\gamma$ is ordered.  Otherwise it is
called a contour with external label {\em dis}, for disordered.

\begin{lemma}
\label{lem:Adding-Cont} Let $A\subset\Om$.  If $\gamma$ is a
contour with $\dist(\bV(A),\Int\gamma)\geq 1/2$ and external
label $\dis$, or a contour with
$\dist(\bV\setminus\bV(A),\Int\gamma)\geq 1/2$ and external
label $\ord$, then there exists a configuration $A'\subset \Om$
such that $\Gamma(A')=\Gamma(A)\cup\{\gamma\}$ and
$S(A)=S(A')$.
\end{lemma}


\begin{proof}
Consider first the case that $\ell=\dis$.  Define $A_1$ to be
the set of edges with both endpoints in $\Int\gamma$.  By
Lemma~\ref{lem:simple} (iii), we have that $\bV(A\cup
A_1)=\bV(A)\cup \bV(A_1)=\bV(A)\cup\Int\gamma$, which shows
that $A\cup A_1$ is the desired configuration.

For $\ell=\ord$, we define $E_1$ to be the set of edges whose
midpoint lies in $\Int\gamma$.  Using Lemma~\ref{lem:simple}
(iv) and (v), we now conclude that $A\setminus E_1$ is the
desired configuration. \qed
\end{proof}

\begin{corollary}
\label{cor:match} Let $A\subset\Om$.  Then the contours and
interfaces corresponding to a configuration $A\in\Om$ are
matching.  Conversely, any set of matching contours and
interfaces corresponds to exactly one configuration $A\in\Om$.
\end{corollary}

\begin{proof}
The first statement is obvious.  The second follows from
Lemmas~\ref{lem:int-base-case} and the previous lemma by
induction on the number of contours. Indeed, by
Lemma~\ref{lem:partial-order}, the partial order on contours
leads to a forest on any set
$\Gamma=\{\gamma_1,\dots,\gamma_n\}$ of pairwise compatible
contours, in such a way that $\gamma'<\gamma$ whenever
$\gamma'$ is a child of $\gamma$.  Adding the interface network
as a common root, we may proceed by induction from this root to
add contours in such a way that the new contour added always
obeys the condition from the previous lemma. \qed
\end{proof}
\medskip

Note that this corollary, together with the representation
\eqref{FK3} for the weights of a configuration $A$, established
the representation \eqref{Z-cont-rep} for the partition
function $Z$.

\section{Key Ingredients for the Lower Bound}
\label{sec:PS-largedev}

 As explained in Section~\ref{sec:proof-strat},
we will prove our lower bound on the mixing time by proving an
upper bound on the conductance. This in turn will require both
a bound on the probability of the set of configurations with at
least one interface, and a large deviation bound on
configurations for which the joint exterior of all contours
contains less than $(1-\alpha) L^d$ points in $V$, see
Lemmas~\ref{lem:Con-sum}  and \ref{mainlem} below.
 We will start
with the decomposition of the partition function
 into terms {\em with} and {\em without} interfaces.

\subsection{Decomposition of the Partition Function}
\label{decomp}

Let
 \begin{equation}
\Omega_\tun =\{A\in\Omega\mid \cS(A)\neq \emptyset\},
\end{equation}
\begin{equation}
\Omega_\ord =\{A\in\Omega\setminus \Omega_\tun
\mid  \Ext\Gamma(A)\subset \bV(A) \}
\end{equation}
and
\begin{equation}
\Omega_\dis =\{A\in\Omega\setminus\Omega_\tun
\mid \Ext\Gamma(A)\subset \bV\setminus\bV(A) \}.
\end{equation}
By Lemma~\ref{lem:extA} (ii),
$\Omega=\Omega_\ord\cup\Omega_\dis\cup\Omega_\tun$. As a
consequence, the partition function {$Z=\sum_{A\in\Omega}
w(A)$ (see (\ref{Z-cont-rep})) can be decomposed as}
\begin{equation}
\label{Z=o+d+tunnel}
Z=Z_\dis +qZ_\ord +Z_\tun,
\end{equation}
where
\begin{equation}
\label{Zodt-def}
Z_\ord=\frac 1q \sum_{A\in\Omega_\ord} w(A),
\quad
Z_\dis=\sum_{A\in\Omega_\dis} w(A)
\quad\text{and}\quad
Z_\tun=\sum_{A\in\Omega_\tun} w(A).
\end{equation}
Note the extra factor of $q$ in \eqref{Z=o+d+tunnel}, which
accounts for the fact that there are $q$ different ordered
phases.

The results of this section are summarized in the next two
lemmas. The first is a finite-size scaling bound analogous to
those proved in \cite{BI,BK,BKM}.  The second is a large
deviations bound.

\begin{lemma}
\label{lem:Con-sum} For all $d\geq 2$, there are constants
$c>0$, $q_0<\infty$ and $L_0<\infty$ such that the following
statements hold for $q\geq q_0$ and $L\geq L_0$:

\noindent {\bf (a)}   If $\beta\geq\beta_0$, then
\begin{equation}
\label{Om-Tun-bd}
\nu(\Omega_\tun )\leq e^{-c\beta L^{d-1}}
\end{equation}
and
\begin{equation}
\label{Om-ord-low-bd}
\nu(\Om_\ord)\geq \frac q{q+1} -e^{-c\beta L}.
\end{equation}

\noindent {\bf (b)} If $\beta=\beta_0$, then
\begin{equation}
\label{Om-ord-equ-bd}
\Bigl|\nu(\Om_\ord) -\frac q{q+1} \Bigr| \leq e^{-c\beta L}.
\end{equation}
\end{lemma}

To state the next lemma, we define
\begin{equation}
\label{Omega_alpha}
\begin{aligned}
\Om_\ord^{(\alpha)} &
=\{A\in \Om_\ord: \,|\Ext\Gamma(A) \cap V|\geq (1-\alpha)L^d\},
\\
\Om_{\dis}^{(\alpha)}
&=\{A\in \Om_{\dis}: \,|\Ext\Gamma(A) \cap V|\geq (1-\alpha)L^d\}.
\end{aligned}
\end{equation}

\begin{lemma}
\label{mainlem} Let $d\geq 2$ and $0<\alpha<1$.  Then there are
constants $c=c(\alpha)>0$ and $q_0=q_0(\alpha)$ such that for
$q\geq q_0$ and $\beta\geq\beta_0$ we have
\begin{equation}
\label{mainlem-a}
\nu\bigl(\Om_\ord\setminus\Om_\ord^{(\alpha)}\bigr)
\leq e^{-c\beta L^{d-1}}
\end{equation}
and
\begin{equation}
\label{mainlem-b}
\nu\bigl(\Om_\dis\setminus\Om_\dis^{(\alpha)}\bigr)
\leq e^{-c\beta L^{d-1}}.
\end{equation}
\end{lemma}


In order  to prove Lemma~\ref{lem:Con-sum}, we will need upper
bounds on $Z_\tun$ and $Z_\dis$, as well as upper and lower
bounds on $Z_\ord$. Since contours and interfaces are
suppressed if $\beta$ (and hence $\kappa$) is large, the
leading configurations to $Z_\ord$ and $Z_\dis$ are those
without contours, giving a contribution of $e^{-e_\ord L^d}$
and $e^{-e_\dis L^d}$, respectively. For $Z_\tun$, the leading
configurations have a single pair of parallel interfaces of
area $L^{d-1}$ each, and no contours, giving a contribution of
at most
 $e^{-2\kappa L^{d-1}}\max\{qe^{-e_\ord L^d},e^{-e_\dis L^d}\}$.
If we took only the leading configurations into account, we
therefore would get that $Z_\tun/Z$ is exponentially suppressed
like $e^{-2\kappa L^{d-1}}$, as required for the first bound in
Lemma~\ref{lem:Con-sum}.  But of course, this is too naive,
since subleading contributions have to be taken into account. A
systematic way to do this is provided by the powerful theory of
Pirogov and Sinai.

\subsection{Ingredients from Pirogov-Sinai Theory}

%
%
%
%


In this section we will prove Lemma~\ref{lem:Con-sum}. To this
end, we will express $Z_\tun$ in terms of partition functions
that are analogs of $Z_\ord$ and $Z_\dis$  for a subset
$\La\subset \bV$ such that
\begin{equation}
\label{form-La}
\La \text{ is a connected component of }
\bV\setminus \partial \bV(A_0)
\text{ for some }A_0\in\Omega.
\end{equation}
We say that a contour $\gamma$ is a contour in $\La$ if
$\dist(V(\gamma),\bV\setminus\La)\geq 1/2$, and we say that a
set of contours $\Gamma$ with matching labels has external
label $\ell\in\{\ord,\dis\}$,  if the external contours in
$\Gamma$ have external label $\ell$. We then set
\begin{equation}
\label{Z-dis-La}
Z_\dis(\La)=\sum_{\Gamma}
q^{c(\bV_\ord)}
e^{-{e_\dis}|\bV_{\dis}\cap V\cap\La|}
e^{-{e_\ord} |\bV_\ord\cap V\cap\La|}
\prod_{\gamma\in\Gamma} e^{-\kappa\|\gamma\|},
\end{equation}
where the sum goes over sets of contours $\Gamma$ with matching
labels such that the external label of $\Gamma$ is $\dis$, and
similarly for $Z_\ord(\La)$:
\begin{equation}
\label{Z-ord-La}
qZ_\ord(\La)=\sum_{\Gamma}
q^{c(\bV_\ord)}
e^{-{e_\dis}|\bV_{\dis}\cap V\cap\La|}
e^{-{e_\ord} |\bV_\ord\cap V\cap\La|}
\prod_{\gamma\in\Gamma} e^{-\kappa\|\gamma\|}.
\end{equation}
{Note that these partition functions are indeed
generalizations of the partition functions $Z_\dis$ and
$Z_\ord$ introduced in \eqref{Zodt-def}. To see this, it is
enough to compare the weights in \eqref{Z-dis-La} and
\eqref{Z-ord-La} to the weight $w(A)$ from \eqref{FK3}, which
shows that $Z_\dis(\bV)=Z_\dis$  and $Z_\ord(\bV)=Z_\ord$.}

With the above definitions, the partition function $Z_\tun$ can
be rewritten as
\begin{equation}
\label{Z-tun-alt}
Z_\tun= \sum_{\cS}\prod_{S\in\cS}
e^{-\kappa \|S\|}
\prod_{\La\in\cC_\dis(\cS)} Z_\dis(\La)
\prod_{\La\in\cC_\ord(\cS)} \Bigl(qZ_\ord(\La)\Bigr),
\end{equation}
where the first sum goes over interface networks, while
$\cC_\ord(\cS)$ is the set of components of
$\bV\setminus\bigcup_{S\in \cS} S$ with $\ell(C)=\ord$, and
$\cC_\dis(\cS)$ is the set of components of
$\bV\setminus\bigcup_{S\in \cS} S$ with $\ell(C)=\dis$. The
formal proof of this, by now almost obvious, identity uses
again the forest structure of sets of pairwise compatible
contours established in Lemma~\ref{lem:partial-order}, and is
similar to that of Corollary~\ref{cor:match}.

\subsubsection{An alternative representation}

We will need a representation for the partition functions
$Z_\dis(\La)$ and $Z_\ord(\La)$ which does not involve the
restriction of matching labels.  To this end, first sum all
terms in \eqref{Z-dis-La} and \eqref{Z-ord-La} which lead to
the same set, $\Gamma_\ext$, of external contours.  Taking into
account  the forest structure of sets of pairwise compatible
contours established in Lemma~\ref{lem:partial-order}, this
leads to the identities
\begin{equation}
\label{Z-ord-alt0}
Z_\ord(\La)
= \sum_{\Gamma_\ext}
e^{-e_\ord|\La\cap\Ext\Gamma_\ext\cap V|}
\prod_{\gamma\in\Gamma_\ext} e^{-\kappa \|\gamma\|}Z_\dis(\Int\gamma)
\end{equation}
and
\begin{equation}
Z_\dis(\La)
= \sum_{\Gamma_\ext}
e^{-e_\dis|\La\cap\Ext\Gamma_\ext\cap V|}
\prod_{\gamma\in\Gamma_\ext} e^{-\kappa \|\gamma\|}qZ_\ord(\Int\gamma),
\end{equation}
where the sums run over sets of mutually external contours in
$\La$ which all have external label $\ord$ and $\dis$,
respectively.

Defining
\begin{equation}
\label{K-ord}
K_\ord(\gamma)= e^{-\kappa \|\gamma\|}\frac{Z_\dis(\Int\gamma)}{Z_\ord(\Int\gamma)},
\end{equation}
we rewrite \eqref{Z-ord-alt0} as
\begin{equation}
\label{Z-ord-ind1}
Z_\ord(\La)
=\sum_{\Gamma_\ext}
e^{-e_\ord|\La\cap\Ext\Gamma_\ext\cap V|}
\prod_{\gamma\in\Gamma_\ext} K_\ord(\gamma){Z_\ord(\Int\gamma)}.
\end{equation}
Inserting \eqref{Z-ord-ind1} inductively into itself, and using
the forest structure of sets of pairwise compatible contours
one last time, we finally arrive at the representation
\begin{equation}
\label{Z-ord-alt}
Z_\ord(\La)
=  e^{-e_\ord|\La\cap V|}\sum_{\Gamma}
\prod_{\gamma\in\Gamma} K_\ord(\gamma)\,,
\end{equation}
where the sum now runs over sets $\Gamma$ of pairwise
compatible contours in $\La$ which all have external label
$\ord$.  In a similar way, one shows that
\begin{equation}
\label{Z-dis-alt}
Z_\dis(\La)
= e^{-e_\dis|\La\cap V|}\sum_{\Gamma}
\prod_{\gamma\in\Gamma} K_\dis(\gamma),
\end{equation}
where the sum  runs over sets $\Gamma$ of pairwise compatible
contours in $\La$ which all have external label $\dis$, and
\begin{equation}
\label{K-dis}
K_\dis(\gamma)= e^{-\kappa \|\gamma\|}\frac{qZ_\ord(\Int\gamma)}{Z_\dis(\Int\gamma)}.
\end{equation}

The representations \eqref{Z-ord-alt} and \eqref{Z-dis-alt}
give $Z_\ord(\La)$ and $Z_\dis(\La)$ as partition functions of
the so-called abstract polymer systems, see, e.g.,
\cite{Brydges} or \cite{B}, for a review. As a consequence, the
logarithms of $Z_\ord(\La)$ and $Z_\dis(\La)$ can be analyzed
by absolutely convergent expansions (so-called
Mayer-expansions), provided the weights \eqref{K-ord} and
\eqref{K-dis} are sufficiently small.

The following lemma gives the bounds needed to apply these
expansions. Given the geometric preparations of the last
section, its proof follows from a careful extension of the
methods of \cite{BK,BKM}.  For the convenience of the reader,
we give it in the appendix.

\begin{lemma}
\label{lem:PS} Let $d\geq 2$.  Then there are  constants
$q_0>0$ and $c>0$, as well as two real-valued functions
$f_\ord=f_\ord(q,\beta)$ and $f_\dis=f_\dis(q,\beta)$ such that
the following statements hold for $\ell\in\{\ord,\dis\}$,
$q\geq q_0$, and $\beta\geq \beta_0$:

(i) Let $f=\min\{f_\ord,f_\dis\}$, and let $a_\ell=f_\ell-f$.
If $\gamma$ is a contour with external label $\ell$ and $a_\ell
\,\diam\gamma\leq c\beta$, then
\begin{equation}
\label{K-bd}
K_\ell(\gamma)\leq e^{-c\beta\|\gamma\|}\,.
\end{equation}

(ii) If $\La\subset \bV$ is of the form \eqref{form-La}, then
\begin{equation}
\label{PS-lb1}
Z_\ell(\La)\geq e^{-(f_\ell+\eps_L)|\La\cap V|}e^{-\|\partial\La\|}
\end{equation}
and
\begin{equation}
\label{PS-up1}
Z_\ell(\La)\leq
e^{(-f+\eps_L)|\La\cap V|}e^{2\|\partial\La\|}
\max_{\Gamma_\ext}e^{-\frac {a_\ell}2|\La\cap\Ext\Gamma_\ext\cap V|}
\prod_{\gamma\in\Gamma_\ext}
e^{-\frac c2\beta\|\gamma\|},
\end{equation}
where the maximum goes over sets of mutually external contours
in $\La$ which all have external label $\ell$, and $\eps_L=2
e^{-c\beta L}$.

(iii) $f_\ord\leq f_\dis$ if $\beta\geq \beta_0$, with equality
if $\beta=\beta_0$.
\end{lemma}

\subsubsection{Proof of Lemma~\ref{lem:Con-sum}}

We start by noting that {by \eqref{beta0-asym} and
\eqref{kappa-def}, the assumption} $\beta\ge \beta_0$ implies
that
$$
{\kappa\geq \frac \beta 2  -\frac 14\geq\frac 1{2d}\log q
- \frac 12}\,,
$$
provided $q$ is large enough (depending on $d$). {We also
recall the notation  $V=V_{L,d}=(\bbZ/L\bbZ)^d$ for the vertex
set of the $d$-dimensional discrete torus of sidelength $L$.}

To prove \eqref{Om-Tun-bd}, we combine \eqref{Z-tun-alt} with
\eqref{PS-up1}, to conclude that
\begin{equation}
\label{Z-tun-bd}
\begin{aligned}
Z_\tun
&\le
 \sum_{\cS}\prod_{S\in\cS}
e^{-\kappa \|S\|}
\prod_{\La\in\cC_\dis(\cS)} e^{(-f{+\eps_L})|\La\cap V|}e^{2\|\partial\La\|}
\prod_{\La\in\cC_\ord(\cS)} qe^{(-f{+\eps_L})|\La\cap V|}e^{2\|\partial\La\|}
\\
&=e^{{(-f+\eps_L)}L^d}\sum_{\cS}\prod_{S\in\cS}
qe^{-(\kappa -4)\|S\|}\,,
\end{aligned}
\end{equation}
where the sum goes over interface networks. In the last step we
used that each interface bounds precisely one ordered and one
disordered component. Using the facts that $q\le e^{2d\kappa}
e^d$, that there are at most
 $2dL(Cd)^k$ interfaces of size $k$
in $V_{L,d}$ (for an appropriate universal constant $C<
\infty$), and that the sum over interface networks contains at
least two interfaces, the bound \eqref{Z-tun-bd} implies that
for $q$ and $L$ large enough (depending on $d$), we have
\begin{equation}
\begin{aligned}
Z_{\tun}
&\le e^{{(-f+\eps_L)}L^d}\sum_{n\ge 2} \Bigl(\sum_{k\ge L^{d-1}}
e^{2d\kappa + d}\, 2dL(Cd)^k e^{ - (\kappa-4)k} \Bigr)^n\\
&={e^{{ (-f+\eps_L)}L^d}\sum_{n\ge 2} \Biggl(e^{2d\kappa + d}\, 2dL
\frac{\bigl(Cde^{ - (\kappa-4)}\bigr)^{L^{d-1}}}
{1-Cd e^{ - (\kappa-4)}}
 \Biggr)^n}\\
&{\leq e^{{(-f+\eps_L)}L^d}\sum_{n\ge 2}e^{-\frac 34\kappa L^{d-1} n}}
\le e^{-fL^d} e^{-\frac{\beta}{2}L^{d-1}}\,,
\end{aligned}
\end{equation}
{where we used that $L^d\eps_L\leq e^{(2d-c\beta)L}$ to
control the factor $e^{\eps_LL^d}$.}

Applying the bound \eqref{PS-up1} to $Z_\dis =
Z_\dis(V_{L,d})$, we get
\[Z_\dis \le e^{-fL^d + L^d\eps_L} \le e^{-fL^d} (1+e^{-c\beta L/2})\,, \]
while the bound \eqref{PS-lb1}, together with the fact that
$f_\ord = f$ if $\beta \ge \beta_0$, gives
\[Z_\ord \ge e^{-fL^d - L^d\eps_L} \ge e^{-f L^d}(1- e^{-c\beta L/2})\,. \]
Together with \eqref{Z=o+d+tunnel}, this gives the first
statement of the lemma.

To prove the second statement, we use that $f=f_\dis = f_\ord$
if $\beta=\beta_0$ which, together with \eqref{PS-lb1} and
\eqref{PS-up1} gives
\[|\log Z_\dis + fL^d| \le L^d\eps_L \le e^{-c\beta L/2} \mbox{ \
and \ }
|\log Z_\ord + fL^d| \le L^d\eps_L \le e^{-c\beta L/2} \,.\]

\subsection{Large Deviation Bounds}
\label{sec:LargeDev}

We now prove Lemma~\ref{mainlem}, by starting  with the proof
of \eqref{mainlem-a}. Let
\[Z_\ord^{<\alpha} = \frac{1}{q} \sum_{A \in \Omega_\ord \colon
\atop{|\Ext\Gamma(A) \cap V| < (1-\alpha)L^d}} w(A)\,,\]
so that
\begin{equation}
\label{main-ord-bd}
\nu(\Om_\ord\setminus \Om_\ord^{(\alpha)}) =
\frac{qZ^{<\alpha}_\ord}{Z} \le \frac{Z^{<\alpha}_\ord}{Z_\ord}\,.
\end{equation}
Proceeding as in the derivation of \eqref{Z-ord-alt}, we
rewrite
\[Z_\ord^{<\alpha} = e^{-e_\ord L^d} \sum_{\Gamma \colon
|\Ext \Gamma\cap V|<(1-\alpha)L^d}\,\, \prod_{\gamma \in \Gamma} K_\ord(\gamma)\,, \]
where the sum runs over sets $\Gamma$ of pairwise compatible
contours in $V_{L,d}$, all of which have external label $\ord$.

{Let $h$ be an arbitrary non-negative number.} Using the
fact that $|\Ext\Gamma \cap V| < (1-\alpha)L^d$ implies that
 $\sum_{\gamma \in \Gamma} |\Int {\gamma} \cap V|\ge \alpha L^d$,
 we then bound
 \begin{equation}
 \label{ord-less}
 Z_\ord^{<\alpha} \le e^{-e_\ord L^d} \sum_\Gamma e^{h(\sum_\gamma |\Int\gamma \cap V|-\alpha L^d)} \
 \, \prod_{\gamma\in \Gamma}K_{\ord}(\gamma)
 = e^{-\alpha h L^d} Z_\ord^{(h)}\,,
\end{equation}
where
\[Z_\ord^{(h)} = e^{-e_\ord L^d} \sum_{\Gamma}  \, \prod_{\gamma\in \Gamma}
e^{h|\Int\gamma \cap V|} K_\ord(\gamma)\,.\]

Next we estimate the dependence of $Z_\ord^{(h)}$ on $h$.  To
this end, we set $K_h(\gamma)=e^{h|\Int\gamma\cap V|}
K_\ord(\gamma)$ and {use a Peierls type argument to bound
the derivative of $Z_\ord^{(h)}$.  Explicitly, we first rewrite
the derivative of $Z_\ord^{(h)}$ as}
\[
\begin{aligned}
\frac{d}{d h} Z_\ord^{(h)}
& = e^{-e_\ord L^d} \,\sum_{\Gamma} \Bigl(\sum_{\gamma\in \Gamma} |\Int\gamma \cap V|\Bigr)
\prod_{\gamma \in \Gamma} K_h(\gamma) \\
& = e^{-e_\ord L^d} \,\sum_{\gamma \subset V_{L,d}} |\Int\gamma \cap V| \sum_{\Gamma \ni \gamma}
 \prod_{\gamma' \in \Gamma} K_h(\gamma')\\
&{=e^{-e_\ord L^d} \,\sum_{\gamma \subset V_{L,d}} |\Int\gamma \cap V|K_h(\gamma)
\sump_{\Gamma'}
 \prod_{\gamma' \in \Gamma'} K_h(\gamma'),}
\end{aligned}
\]
{where the sum over $\gamma \subset V_{L,d}$ denotes a
sum over contours in $V_{L,d}$ with external label $\ord$, and
the sum $\sump$ denotes a sum over sets $\Gamma'$ of pairwise
compatible contours in $V_{L,d}$ such that
\begin{itemize}
\item  all contours $\gamma'\in\Gamma$ have external label
    $\ord$;
\item all contours $\gamma'\in\Gamma$ are compatible with
    $\gamma$.
\end{itemize}
Removing the second constraint, we bound this sum by
$$
\sump_{\Gamma'}
 \prod_{\gamma' \in \Gamma'} K_h(\gamma')\leq
 \sum_{\Gamma}
 \prod_{\gamma' \in \Gamma} K_h(\gamma')
 =e^{e_\ord L^d} Z_\ord^{(h)},
$$
which shows that}
\[
\begin{aligned}
\frac{d}{d h} \log Z_\ord^{(h)}
&=\frac 1{Z_\ord^{(h)}}\frac{d}{d h} Z_\ord^{(h)}
\le \sum_{\gamma \subset V_{L,d}} |\Int\gamma \cap V|  K_h(\gamma) \\
&\le \sum_{\gamma \subset V_{L,d}}{ \frac14 \|\gamma\|^2} \, e^{h\|\gamma\|L} e^{-c \beta \|\gamma\|}\,,
\end{aligned}
\]
where we used the bounds {\eqref{g-larger-diam-g}},
\eqref{iso1}  and \eqref{K-bd}, together with the fact that
$f_\ord = f$ for $\beta \ge \beta_0$, in the last step.

Now assume that  $h\le c\beta/(2L)$. Since there are most
$dL^d(Cd)^k$ contours of size $k$ in $V_{L,d}$, we conclude
that
\[
\frac{d}{d h} \log Z_\ord^{(h)}
\le \frac12 \sum_{k\ge 2}dL^d(Cd)^k {k^2} e^{- c\beta k/2} \le \frac{\alpha}{4}L^d\,,
\]
and thus
\[Z_\ord^{(h)} \le Z_\ord \, e^{\alpha hL^d/4}\,,\]
provided $\beta$ is large enough (depending on $d$ and
$\alpha$.) Inserted into \eqref{ord-less} and
\eqref{main-ord-bd}, this gives
\[\nu(\Om_\ord \setminus \Om_\ord^{\alpha}) \le e^{-3\alpha hL^d/4}
 = e^{-3c\alpha \beta L^{d-1}/8}\,,\]
where we have set $h$ equal to $c\beta/(2L)$ in the last step.

This proves \eqref{mainlem-a}. For $a_\dis L \le c\beta$, the
bound \eqref{mainlem-b} is proved in exactly the same way, but
for $a_\dis L > c\beta$, this strategy does not work, since
\eqref{K-bd} is not at our disposal anymore. In stead, we use
\eqref{PS-lb1}, \eqref{PS-up1}  and \eqref{iso1}  to bound
\[
\begin{aligned}
\nu(\Om_\dis \setminus \Om_\dis^{(\alpha)}) \le \nu(\Om_\dis)
&\le \frac{Z_\dis}{Z_\ord} \le e^{2L^d\eps_L} \max_{\Gamma_\ext}
 e^{- \frac{a_\dis}2|\Ext\Gamma_\ext \cap V|} \,
\prod_{\gamma \in \Gamma_\ext} e^{-\frac c2\beta\|\gamma\|}\\
& \le e^{2L^d\eps_L} \max_{\Gamma_\ext}
e^{- \frac{c\beta}{2L}|\Ext\Gamma_\ext \cap V|}
\, \prod_{\gamma\in \Gamma_\ext} e^{-\frac{c\beta}{L} |\Int\gamma \cap V|}\\
& \le e^{2L^d\eps_L} e^{-\frac{c\beta}{2L} L^d}\,,
\end{aligned}
\]
where we used that $|\Ext \Gamma_\ext \cap V| = L^d -
\sum_{\gamma \in \Gamma_\ext}
 |\Int\gamma \cap V|$ in the last step.
 This concludes the proof of the lemma.

\section{Lower bounds on the mixing time}
\label{sec:LowerBound}

Lemma~\ref{lem:Con-sum} and Lemma~\ref{mainlem} give us the
necessary ingredients to prove our main result.  We start by
proving a key theorem which expresses the statements of these
lemmas in terms of the probability measure $\mu$ on spin
configurations, rather than the {random cluster} measure
$\nu$.

\subsection{An Important Probabilistic Estimate}

Recall that for a spin configuration $\bss\in [q]^V$, we
defined $E(\bss)$ to be the set of all edges whose two
endpoints have the same color. We also introduce the set,
$\cC(\bss)$, of connected components of the graph
$(V,E(\bss))$, that is,  the set of monochromatic components.

\begin{theorem}
Let $d\geq 2$ and $0<\alpha<1/2$.  Then there are constants
$c>0$, $q_0<\infty$ and $L_0<\infty$ such that for $q\geq q_0$
and $L\geq L_0$ the following statements hold:

\noindent {\bf (a)}   If $\beta=\beta_0$, then
\begin{equation}
\label{SW-input1}
\mu\Bigl(\alpha d L^d{<}|E(\bss)|
{<} (1-\alpha)dL^d\Bigr)
\leq e^{-c\beta L^{d-1}}
\end{equation}
and
\begin{equation}
\label{SW-input2}
\Bigl|\mu\Bigl(|E(\bss)|\geq (1-\alpha)dL^d\Bigr)
-\frac q{q+1}\Bigr|\leq e^{-c\beta L}.
\end{equation}

\noindent {\bf (b)} If $\beta\geq\beta_0$, then
\begin{equation}
\label{HB-input1}
\mu\biggl(\alpha L^d<\max_{C\in\cC(\bss)}|V(C)|
<(1-\alpha)L^d\biggr)
\leq e^{-c\beta L^{d-1}}
\end{equation}
and
\begin{equation}
\label{HB-input2}
\mu\biggl(\max_{C\in\cC(\bss)}|V(C)|\geq (1-\alpha)L^d\biggr)
\geq
\frac q{q+1} -
e^{-c\beta L}.
\end{equation}
\end{theorem}

\begin{proof}
To relate the statements of Lemmas~\ref{lem:Con-sum} and
\ref{mainlem} to the theorem, we use that both the spin measure
$\mu$ and the FK-measure $\nu$ are marginals of the
Edwards-Sokal measure $\pi$.  Consider thus a configuration
$(\bss,A)$ with positive measure $\pi((\bss,A))$.  Under this
condition, all spins in a component of $(V,A)$ must have the
same color, implying in particular that
\begin{equation}
\label{A<E(bss)}
A\subset E(\bss)\,.
\end{equation}
It turns out that with high probability, $|A|$ is not much
smaller than $|E(\bss)|$ either.  More precisely, we will prove
that
\begin{equation}
\label{E(bss)>A}
\pi\Bigl(|E(\bss)|\geq |A|+{\tilde\alpha} dL^d\Bigr)
\leq {e^{-(\tilde\alpha \beta-1) dL^{d}}
\quad\text{for all}\quad
\tilde\alpha\in (0,1).}
\end{equation}
We will also show that, again under the condition that
$\pi((\bss,A))>0$,
\begin{equation}
\label{C>C}
\max_{C\in \cC(\bss)}|V(C)|\geq \max_{C\in \cC(V,A)}|V(C)|
\end{equation}
and
\begin{equation}
\label{C<C}
\max_{C\in \cC(\bss)}|V(C)|\leq \max_{C\in \cC(V,A)}|V(C)| + |E(\bss)\setminus A|+1.
\end{equation}
As we will see below, these bounds, together with
Lemmas~\ref{lem:Con-sum} and \ref{mainlem}, imply the
statements of the theorem.

Before showing this, we will prove \eqref{E(bss)>A} --
\eqref{C<C}. We start with the proof of \eqref{E(bss)>A}. To
this end, we rewrite the left hand side as
\[
\pi\Bigl(|E(\bss)|\geq |A|+\tilde\alpha dL^d\Bigr)
=\sum_{\bss}
\pi\Bigl(|E(\bss)|\geq |A|+\tilde\alpha dL^d\,\Big|\,\bss\Bigr)
\mu(\bss).
\]
But given $\bss$, a configuration $A$ according to the
conditional measure $\pi(\cdot\mid\bss)$ is  obtained by
deleting the edges in $E(\bss)$ independently with probability
$e^{-\beta}$.  The number of deleted edges is therefore equal
to a binomial random variable with parameters $m$ and
$e^{-\beta}$, where $m=|E(\bss)|\leq dL^d$. We now bound the
probability that the number of deleted edges $X$ is larger than
$\tilde\alpha dL^d$ as follows:
\begin{equation}
\label{simplebound}
\Pr (X \geq \tilde\alpha dL^d) = \sum_{k\geq\tilde\alpha dL^d} \binom{m}{k}
e^{-\beta k} (1-e^{-\beta})^{m-k} \leq e^{-\beta\tilde\alpha dL^d}2^{dL^d}.
\end{equation}
This implies the bound \eqref{E(bss)>A}.

{Next, we observe that \eqref{C>C} follows from
\eqref{A<E(bss)}. To prove \eqref{C<C}, we first show that any
for  $C \in \cC(\bss)$ and any  $D\subset E(C)$, we have:
\begin{equation}
\label{cycle_space}
 |D| - |V(D)| \le |E(C)| - |V(C)|+1\,.
\end{equation}
The simplest (albeit non-elementary) way to see this is by
recalling that the dimension of the  cycle subspace of the edge
space of a graph on $m$ edges and $n$ vertices and $k$
components equals $m-n+k$; we view here the collection of all
edges as the $m$-dimensional vector space over $GF(2)$ and
consider the edge sets of all simple cycles as a subspace of
it. For non-empty $D$, this gives in fact the stronger bound
$$
 |D| - |V(D)|+1 \le |E(C)| - |V(C)|+1\,.
$$
If $D=\emptyset$, the bound (\ref{cycle_space}) is trivial,
since the right hand side of (\ref{cycle_space}) is
non-negative by the fact that $C$ is connected.

Consider now a component $\tilde C$ of $(V(A),A)$.  Since
$A\subset E(\bss)$, there must be a component $C\in\cC(\bss)$
such that $V(\tilde C)\subset V(C)$ and $E(\tilde C)\subset
E(C)$. Applying (\ref{cycle_space}) with $D=E(\tilde C)$, we
get
$$
\begin{aligned}
|V(C)|  &\le |V(\tilde C)|+ |E(C)| -|E(\tilde C)|  +1
\\
&\leq |V(\tilde C)| +|E(\bss)|-|A|+1
\\
&\leq \max_{C'\in\cC(V,A)}|V(C')| +|E(\bss)|-|A|+1.
\end{aligned}
$$
This in turn implies \eqref{C<C}.}

We are now ready to prove statement (a). To this end, consider
a configuration $A\in\Om_\ord^{(\alpha)}$, {where
$\Om_\ord^{(\alpha)}$ (and $\Om_\dis^{(\alpha)}$) are as
defined in (\ref{Omega_alpha})}.   Then {$\Ext\Gamma(A)\cap V$ is connected set by Lemma~\ref{lem:extA}.
If $\bss$ is such that $\pi((\bss,A))>0$, all edges joining two
points in $\Ext\Gamma(A)\cap V$ must then be part of $E(\bss)$.
Since the number of edges intersecting the  complement of
$\Ext\Gamma(A)\cap V$ is at most $2d(L^d-|\Ext\Gamma(A)\cap
V|)\leq 2d\alpha L^d$, we concluded that $E(\bss)$ contains at
least $dL^d-2d\alpha L^d$ edges.
%
%
%
%
In summary
\begin{equation}
\label{simple-impl1}
A\in\Om_\ord^{(\alpha)}\quad\text{and}\quad\pi((\bss,A))>0
\implies
|E(\bss)|\geq (1-2\alpha)dL^d.
\end{equation}
} On the other hand, by the fact that $|A|\leq d|V(A)|$ for all
$A\subset E$, we have that
\begin{equation}
\label{simple-impl2}
A\in\Om_\dis^{(\alpha)}
\implies
|V\setminus V(A)|\geq (1-\alpha)L^d
\implies
|A|\leq \alpha dL^d.
\end{equation}

We now turn to the first bound of the theorem. {To this
end, we use \eqref{simple-impl1}, \eqref{E(bss)>A}, and
\eqref{simple-impl2} to bound the left hand side of
\eqref{SW-input1} by
$$
\begin{aligned}
\nu(\Omega_\tun)
&+
\pi\Big(A\in\Omega_\ord\text{ and } |E(\bss)|< (1-\alpha)dL^d\Big)
+
\pi\Big(A\in\Omega_\dis\text{ and }|E(\bss)|>\alpha dL^d\Big)
\\
&\leq
\nu(\Omega_\tun)+\nu(\Omega_\ord\setminus\Omega_\ord^{(\alpha/2)})
+\nu(\Omega_\dis\setminus\Omega_\dis^{(\alpha/2)})
+e^{-(\frac12\alpha \beta -1){d}L^{d}}.
\end{aligned}
$$
Bounding the terms on the right hand side with the help of
\eqref{Om-Tun-bd}, \eqref{mainlem-a} and \eqref{mainlem-b},} we
therefore obtain that there exists a constant $c>0$ depending
on $\alpha$ and $d$ such that
\[
\mu\Bigl(\alpha d L^d<|E(\bss)|<(1-\alpha)d L^d\Bigr)
\leq 4e^{-{2}c\beta L^{d-1}}
\leq e^{-{c}\beta L^{d-1}}\,,
\]
provided $q$ (and hence $\beta$) is large enough. This proves
the bound \eqref{SW-input1}.

The proof of the bound \eqref{SW-input2} is similar. {Indeed, starting again with the implication
\eqref{simple-impl1},
we have
\[
\begin{aligned}
\mu\Bigl(|E(\bss)|&\geq (1-\alpha)d L^d\Bigr)
\geq
\nu\Bigl(\Om_\ord^{(\alpha/2)}\Bigr)
\geq
\frac{q}{q+1} - e^{-c\beta L^{d-1}} -e^{-c\beta L},
\end{aligned}
\]
where we used} the  bounds \eqref{Om-ord-low-bd} and
\eqref{mainlem-a} in the last step.

{On the other hand, by \eqref{E(bss)>A},
\[
\begin{aligned}
\mu\Bigl(|E(\bss)|&\geq (1-\alpha)d L^d\Bigr)
\leq
\nu\Bigl(|A|> (1-2\alpha) dL^d\Bigr)+e^{-(\alpha\beta-1)dL^d}
\\
&\leq
\nu(\Omega_\tun)+\nu(\Omega_\ord^{(2\alpha)})
+\nu(\Omega_\dis\setminus\Omega_\dis^{(1-2\alpha)})
+e^{-(\alpha\beta-1)dL^d}.
\end{aligned}
\]
Combined with the bounds \eqref{Om-Tun-bd},
\eqref{Om-ord-equ-bd}, \eqref{mainlem-a} and \eqref{mainlem-b},
this provides a matching upper bound on $\mu\bigl(|E(\bss)|\geq
(1-\alpha)d L^d\bigr)$, completing the proof of
\eqref{SW-input2}, and hence of part (a).}
%
%
%

The bounds of part (b) are proved in a similar way. Indeed, let
$A\in\Om_\ord^{(\alpha)}$ and let $A_\ext$ be the set of edges
with both endpoints in $\Ext\Gamma(A)$. By Lemma~\ref{lem:extA}
(iii), the graph $(V(A_\ext),A_\ext)$ is connected, implying
that $\max_{C\in\cC(V,A)}|V(C)|\geq (1-\alpha)L^d$ whenever
$A\in\Om_\ord^{(\alpha)}$. Taking into account the  bound
\eqref{C>C} we get that
\begin{equation}
\label{simple-impl3}
A\in\Om_\ord^{(\alpha)}\text{ and }\pi((\bss,A))>0
\implies
\max_{C\in\cC(\bss)}|V(C)|\geq (1-\alpha)L^d\,.
\end{equation}
Together with \eqref{Om-ord-low-bd} and \eqref{mainlem-a}, this
immediately gives the bound \eqref{HB-input2}.

We are thus left with the proof of the bound \eqref{HB-input1}.
To this end, we note that if $|V(A)|\leq \frac 12\alpha L^d$
and $|E(\bss)|{<} |A|+\frac 12\alpha L^d{-1}$ then,
by (\ref{C<C}),  the largest component of $(V,E(\bss))$ has
number of vertices {less than $|V(A)| + |E(\bss)| -
|A|+1$, which in turn is less than}
$\alpha L^d$ . Combined with the bound \eqref{E(bss)>A}, we
conclude that
\[
\pi\Bigl(A\in\Om_\dis^{(\alpha/2)}
\text{ and }\max_{C\in\cC(\bss)}|V(C)|\geq \alpha L^d\Bigr)
\leq
{e^{-(\tilde\alpha \beta -1)dL^{d}}},
\]
where $\tilde\alpha = \frac{\alpha}{2d}-\frac 1{dL^d}$.
Combined with \eqref{simple-impl3}, this implies that the left
hand side of \eqref{HB-input1} can be bounded by
\[
\mu\Bigl(\Om\setminus
(\Om_\ord^{(\alpha)}\cup\Om_\dis^{(\alpha/2)})\Bigr)
+{e^{-(\tilde\alpha \beta -1)dL^{d}}}.
\]
Together with the bounds \eqref{Om-Tun-bd}, \eqref{mainlem-a}
and \eqref{mainlem-b}, this gives the desired bound
\eqref{HB-input1}. \qed
\end{proof}

\subsection{Proof of the SW bound in Theorem \ref{thm:SW}}
Let $\beta=\beta_0$. Let $S=\{\bss : |E(\bss)| \ge
(1-\alpha)dL^d\}$. We will show that $\Phi_S$
 is exponentially small in $\beta L^{d-1}$, which will establish the theorem. For $q$
(and hence $\beta_0$) large enough, $\pi(S) \ge 1/2$, using
\eqref{SW-input2}. Also, $\pi(S^c) \ge 1/q - e^{-c\beta L} \ge
1/2q$, if $L$ is large enough. Thus
\begin{equation}
\label{SW-cond-bound}
\Phi_S = \frac{Q(S,S^c)}{\mu(S)\mu(S^c)} \le 4q Q(S,S^c).
\end{equation}
Let $S_0 = \{\bss : \alpha dL^d {<}|E(\bss)| {<}
(1-\alpha){d}L^d\}$. Then
\begin{equation}
Q(S,S^c) = Q(S, S^c \setminus S_0) + Q(S, S_0).
\end{equation}
Now
\[ Q(S,S_0) = Q(S_0,S) \le \pi(S_0) \le e^{-c\beta L^{d-1}},\]
using \eqref{SW-input1}, while
\[Q(S, S^c \setminus S_0) = \pi(S) \Pr\bigl({|}E(\bss'){|\leq}
 \alpha dL^d \, | \, {|}E(\bss) {|}
 \ge (1-\alpha)dL^d\bigr),\]
where $\bss$ is a $\mu$-random spin configuration, and $\bss'$
is constructed from one step of the SW algorithm. The above
probability is in turn the probability that at least
$(1-2\alpha)dL^d$ edges are deleted in one step of the SW
algorithm, which is at most $\bigl(2e^{-\beta
(1-2\alpha)}\bigr)^{dL^d}$, using \eqref{simplebound}. Choosing
$\alpha=1/3$, and $q$ large (so that $\beta$ is large)
 and $L$ sufficiently large, yields:
\[Q(S, S^c \setminus S_0) \le \pi(S)\bigl(2e^{-\beta(1-2\alpha)}\bigr)^{dL^d}
\le e^{-c\beta L^{d-1}}.\]
This implies the desired bound on conductance,
\[ \Phi(P^{\rm SW}) \le e^{-c\beta L^{d-1}},\]
which together with \eqref{mix-cond-lb} concludes the proof of
the lower bound \eqref{SW-tau-low} on the mixing time of the SW
algorithm. \qed

\subsection{Proof of the HB bound in Theorem \ref{thm:SW}}

Let $\bss\in [q]^V$ and let $C$ be a component of
$(V,E(\bss))$.  We say that $C$ has color $k$, if $\sigma_x=k$
for all $k\in V(C)$, and denote the set of components of color
$k$ by $\cC_k(\bss)$. For $k=1, \dots, q$ and $0<\alpha<1/2$,
we then define
\[\widehat\Omega_k^{(\alpha)}
=\Bigl\{\bss\in [q]^V\colon
\exists \;C\in\cC_k(\bss)
\;\text{s.t.}\;
|V(C)|\geq (1-\alpha)|V|
\Bigr\}.
\]
Note that the sets $\widehat\Omega_k^{(\alpha)}$ are mutually
disjoint, so by symmetry and the bound \eqref{HB-input2}, we
have
\begin{equation}
\mu\bigl(\widehat\Omega_k^{(\alpha)}\bigr)
\geq
\frac 1{q+1} -\frac 1q
e^{-c\beta L}\,.
\end{equation}
Finally, let
\[
\widehat\Om_\dis^{(\alpha)}
=\Bigl\{\bss\in [q]^V\colon
\max_{C\in \cC(\bss)}|V(C)|\leq \alpha L^d\Bigr\}\,.
\]
We complete our proof by estimating $\Phi_S$ (see (\ref{cond}))
for $S=\widehat\Omega_1^{(\alpha)}$. First notice that for $q$
(and hence also $\beta$) sufficiently large
$\mu(S)\mu({S^c})\geq  1/4q$ so that
\[
\Phi_S\leq 4qQ(S,S^c).
\]
 Since the
heat bath algorithm can only change one vertex at a time, it
does not make transitions between the different sets
$\widehat\Omega_k^{(\alpha)}$. For $\alpha$ small enough, it
cannot make transitions between $\widehat\Omega_1^{(\alpha)}$
and $\Omega^{(\alpha)}_\dis$ either.  Indeed, changing the
color of a single vertex can not break a component $C\in
\cC(\bss)$ into more than $2d+1$ new components: in the worst
case, $C$ gets broken into a single component of size one and
$2d$ components of size $((1-\alpha)L^d-1)/(2d)$. For $\alpha$
sufficiently small (say $\alpha=1/(4d)$), the heat bath
algorithm therefore cannot make transitions between
$\widehat\Omega_1^{(\alpha)}$ and $\Omega^{(\alpha)}_\dis$.
Defining $S_0$ as the set of configurations which are neither
in $\Omega^{(\alpha)}_\dis$ nor in one of the sets
$\widehat\Omega_k^{(\alpha)}$, we thus have
\[
Q(S,S^c)=Q(S,S_0)=Q(S_0,S)\leq \mu(S_0)\leq e^{-c\beta L^{d-1}},
\]
where we have used the bound \eqref{HB-input1} in the last
step. Recalling that $\beta\geq\beta_0=\frac 1d\log
q+O(q^{-1/d})$, we see that for $L$ sufficiently large, we have
$\Phi_S\leq 4q Q(S,S^c)\leq e^{-c\beta L^{d-1}/2}$, as
required.  \qed


\bigskip

\noindent {\bf Acknowledgements:}  The authors wish to thank
Marek Biskup, Michael Freedman, Roman Koteck\'y, Fabio
Martinelli and Jacob Lurie for numerous helpful discussions.
The authors also thank their collaborators Alan Frieze, Jeong
Han Kim, Eric Vigoda, and Van Vu for their efforts on earlier
research (reported in \cite{confversion}) related to the
present contribution. {The authors are also grateful to
an anonymous referee  for a very careful review and helpful
suggestions which resulted in an improved exposition.}

\bibliographystyle{latex8}

\appendix

\section{Technical Estimates using Truncation}

Throughout this section, we will assume that
\begin{equation}
{\beta}\geq \max\Bigl\{{C_1\log(dC)},{\frac 1d}\log q-{1}\Bigr\},
\label{kappa-asum2}
\end{equation}
{where $C$ is the constant from Lemma~\ref{lem:cont-estimates}
and $C_1$ is a suitable constant to be chosen in the
course of the proof.}
In fact, we will prove statements (i) and
(ii) of Lemma~\ref{lem:PS} for all $\beta$ such that
\eqref{kappa-asum2} holds {(whether $\beta\geq \beta_0$ or not).}

Also, for the purpose of this appendix, we will use the symbol
$|\La|$ for the cardinality of the set $V\cap\La$, so that
expressions of the form $|V\cap\Ext\Gamma\cap\La|$ can be
simplified to $|\La\cap\Ext\Gamma|$.

\subsection{Truncated Contour Models}

We need some preparation.
We start by bounding the factor
$qe^{-\kappa\|\gamma\|}$ appearing in the weight \eqref{K-dis}
of a contour with external label $\dis$.  To this end, we
observe that the smallest contour with disordered external
label has size $\|\gamma\|\geq 4d-2$.  Combined with the
assumption \eqref{kappa-asum2} and the assumption $d\geq 2$,
this gives
\begin{equation}
\label{raw-K-bd}
\begin{aligned}
qe^{-\kappa\|\gamma\|}
&\leq e^{{d(\beta+1)}}e^{-\kappa\|\gamma\|}
\leq\exp\Bigl(\Bigl(\frac {d\beta+d}{4d-2}-\kappa\Bigr)\|\gamma\|\Bigr)
\leq e^{-{\frac\beta 8} \|\gamma\|}\,,
\end{aligned}
\end{equation}
 {where in the last step we assume that $C_1$ is chosen large enough to
 guarantee that  $\kappa=\frac \beta 2 +O(e^{-\beta})\geq  \frac \beta 8 + \frac{1+\beta}3$ .}

As usual in Pirogov-Sinai theory, we next introduce a truncated
model. It is given in terms of the truncated activities
\begin{equation}
\label{K-tr}
K_\ell'(\gamma)
=\min\{K_\ell(\gamma), e^{- {(\frac \beta 8 - c\beta+1)}\|\gamma\|}\}
\end{equation}
 and the corresponding partition functions
\begin{equation}
\label{Z-tr}
Z_\ell'(\La)
=  e^{-e_\ell|\La|}\sum_{\Gamma}
\prod_{\gamma\in\Gamma} K_\ell'(\gamma),
\end{equation}
where $\ell\in\{\ord,\dis\}$, the sum in \eqref{Z-tr} runs over
sets $\Gamma$ of pairwise compatible contours in $\La$ which
all have external label $\ell$, and $c$ is a small enough constant;
{we will choose $c=1/{20}$, implying in particular that
$K_\ell'(\gamma)\leq e^{-c\beta\|\gamma\|}$.}

Let $x\in V$,  let $\gamma_0$ be a contour or an interface, let
$\ell\in\{\ord,\dis\}$. With the help of
Lemma~\ref{lem:cont-estimates}, we then bound
\begin{equation}
\sum_{\gamma:\Int\gamma\ni x}K'_\ell(\gamma)e^{(c\beta+1)\|\gamma\|}
\leq
{\sum_{\gamma:\Int\gamma\ni x}
e^{-( {\frac\beta 8-2c\beta})\|\gamma\|}}
{\le
\sum_{k\geq 1}e^{-( {\frac\beta 8-2c\beta})k}(Cd)^k}
\le 1,
\end{equation}
and
\begin{equation}
\sum_{\gamma:\dist(\gamma,\gamma_0)<1/2}K'_\ell(\gamma)e^{(c\beta+1)\|\gamma\|}
{\le\|\gamma_0\|
\sum_{k\geq 1}e^{-( {\frac\beta 8-2c\beta})k}(Cd)^k}
\le \|\gamma_0\|,
\end{equation}
{provided $C_1$ is sufficiently large.} The above bounds imply
absolute convergence of the cluster expansions for abstract
polymer systems, which in turn gives the existence of the
limits
\begin{equation}
\label{f-prime}
f_\ell=\lim_{L\to\infty} f_\ell^{(L)}
\quad\text{with}\quad
f_\ell^{(L)}=-\frac 1{L^d}\log Z'_\ell(V_{L,d})\,,
\end{equation}
where $\ell\in\{\ord,\dis\}$ and
$V_{L,d}$ denotes the $d$-dimensional torus of sidelength $L$,
see, e.g., \cite{Brydges,B} for a review of cluster expansions
for abstract polymer systems. These methods also imply that,
for $\ell\in\{\ord,\dis\}$, and $\La\subset \bV$ of the form
\eqref{form-La}, we have $|f_\ell-f_\ell^{(L)}|\leq \eps_L$ and
\begin{equation}
\label{Z-cluster-bd}
\Bigl|
\log Z_\ell'(\La)+f_\ell |\La|
\Bigr|
\leq \|\partial\La\|+\eps_L|\La|\,,
\end{equation}
where, as before, $\eps_L=2e^{-c\beta L}$.
  We will assume that $C_1$ in \eqref{kappa-asum2}
has been chosen in such a way that $L\eps_L\leq 1$.

\subsection{Proof of Lemma~\ref{lem:PS} (i) and (ii)}

We first note that $Z_\ell(\La)\ge Z_\ell'(\La)$, so in view of
\eqref{Z-cluster-bd}, we have
\begin{equation}
\label{Zell-low-bd}
Z_\ell(\La)\geq e^{-(f_\ell+\eps_L)|\La|}e^{-\|\partial \La\|}.
\end{equation}
{Next we recall that
 $K'_\ell(\gamma)\leq e^{-c\beta\|\gamma\|}$.  As a consequence
 $K_\ell(\gamma)\leq e^{-c\beta\|\gamma\|}$
whenever $K_\ell(\gamma)=K'_\ell(\gamma)$}.  To prove
Lemma~\ref{lem:PS} (i) and (ii), it is therefore enough to
establish the following lemma:
\begin{lemma}
\label{lem:PS-A1} Under the condition \eqref{kappa-asum2}, we
have that

(i) $ K_\ell(\gamma)=K'_\ell(\gamma)$ whenever $\gamma$ is a
contour with external label $\ell$ and $a_\ell
\,\diam\gamma\leq c\beta$.

(ii) For all $\La$ of the form \eqref{form-La},
\begin{equation}
\label{PS-up2}
\begin{aligned}
Z_\ell(\La)
\leq
e^{(\eps_L-f)|\La|}e^{2\|\partial\La\|}
\max_{\Gamma_\ext}e^{-\frac {a_\ell}2|\La\cap\Ext\Gamma_\ext|}
\prod_{\gamma\in\Gamma_\ext}
e^{-\frac c2\beta\|\gamma\|}\,,
\end{aligned}
\end{equation}
where the maximum goes over sets of mutually external contours
in $\La$ which all have external label $\ell$.
\end{lemma}

\begin{proof}
We prove the lemma by induction on the levels of $\La$ and
$\gamma$:  Here the level of a set $\La$ is defined to be zero
if there are no contours $\gamma$ such that $\gamma$ is a
contour in $\La$. The level of a set $\La$ is defined to be $k$
if the highest level of a contour $\gamma$ in $\La$ is $k-1$,
with the level of a contour inductively defined to be $1$ plus
the level of its interior. Note that with this definition, the
levels of two contours $\gamma,\gamma'$ with $\gamma<\gamma'$
differ by at least two.

Assume that $\La$ has level $0$.  Recall the definition of $e_\ell$  from \eqref{e-dis-def} and \eqref{e-ord-def}. Then $Z_\ell(\La)=e^{-e_\ell |\La|} \leq e^{-f_\ell|\La|}= e^{-f|\La|} e^{-a_\ell|\La|}$, where the inequality may be seen as follows: first for a fixed $L$, consider $Z_\ell^\prime$ with $\Lambda$ being the entire torus of size $L^d$; the term with no contours gives a contribution $e^{-e_\ell L^d}$, implying the desired bound for $f_l^{(L)}$.  Taking the limit $L\to \infty$, one gets the inequality.
This proves \eqref{PS-up2} for sets $\La$ of level $0$,  the base case.

Next assume that the bound \eqref{PS-up2} in (ii) has been
proven for all sets $\La$ of level $k$ or less.  If $\gamma$
has level $k+1$ or less {and label $\ell=\dis$}, then
\[
\begin{aligned}
K_\dis(\gamma)
&\leq e^{-{\frac\beta 8}\|\gamma\|}
\frac{Z_\ord(\Int\gamma)}{Z_\dis(\Int\gamma)}
\\
&\leq e^{(a_\ell+2\eps_L)|\Int\gamma|}e^{-(\frac\beta8-3)\|\gamma\|}
\leq
e^{-(\frac\beta8-3-\frac12(a_\ell+2\eps_L)\diam\gamma)\|\gamma\|}\,,
\end{aligned}
\]
where we used \eqref{raw-K-bd} in the first inequality, the
inductive assumption \eqref{PS-up2} and the bound
\eqref{Zell-low-bd} in the second, and the bound \eqref{iso1}
in the last.  Bounding $\eps_L\diam\gamma$ by $L\eps_L\leq 1$
and using the assumption $a_\ell \,\diam\gamma\leq c\beta$,
this gives $K_\dis(\gamma)\leq
e^{-({\frac\beta 8}-4-{\frac c2}\beta)\|\gamma\|}$ and hence
$K_\dis(\gamma)=K'_\dis(\gamma)$ {(again provided $C_1$ is sufficiently large)}.
The bound for contours with
ordered external label is exactly the same.

Finally, assume that $\La$ has level $k+2$, that (i) has been
proven for all contours of level at most $k+1$, and that (ii)
has been proven for all sets of level at most $k$. Define a
contour $\gamma$  with external label $\ell$ to be \emph{small}
if $a_\ell \,\diam\gamma\leq c\beta$, and \emph{large}
otherwise. Consider the representation \eqref{Z-dis-La} for
$Z_\dis(\La)$, and fix, for a moment, the set
$\Gamma_{\text{large}}$ of all large external contours
contributing to the right hand side. Summing over the remaining
contours, we get a factor of $Z_\ord(\Int\gamma)$ for the
interior of each contour $\gamma\in\Gamma_{\text{large}}$, as
well as a factor
$Z_{\dis}^{(\text{small})}(\Ext\Gamma_{\text{large}})$ for the
exterior of $\Gamma_{\text{large}}$, where
$Z_\dis^{(\text{small})}(\La')$ is obtained from $Z_\dis(\La')$
by dropping all configurations with large external contours.
Thus
\begin{equation}
\label{Z-dis-l-s}
Z_\dis(\La)
=\sum_{\Gamma_{\text{large}}}
Z_{\dis}^{(\text{small})}(\Ext\Gamma_{\text{large}})
\prod_{\gamma\in\Gamma_{\text{large}}}
qe^{-\kappa\|\gamma\|} Z_\ord(\Int\gamma)\,,
\end{equation}
where the sum goes over sets of mutually external, large
contours with disordered external label.

Since $\gamma$ is small whenever $\gamma<\gamma'$ and $\gamma'$
is a {small} contour with the same external label as $\gamma$, the
representation \eqref{Z-dis-alt} for
$Z_{\dis}^{(\text{small})}(\Ext\Gamma_{\text{large}})$ contains
only small contours, implying that for all these contours
$K_\dis(\gamma)=K_\dis'(\gamma)$.  As a consequence,
$$Z_{\dis}^{(\text{small})}(\Ext\Gamma_{\text{large}})
\leq Z_{\dis}'(\Ext\Gamma_{\text{large}})\,,$$ which allows us to use the
estimate \eqref{Z-cluster-bd} to estimate the factor
$Z_{\dis}^{(\text{small})}(\Ext\Gamma_{\text{large}})$ in
\eqref{Z-dis-l-s}.  Using the inductive assumption (ii) to
bound the factors $Z_\ord(\Int\gamma)$ and the bound
\eqref{raw-K-bd} to estimate the factors
$qe^{-\kappa\|\gamma\|}$, this gives
\[
\begin{aligned}
Z_\dis(\La)&\leq
\sum_{\Gamma_{\text{large}}}
e^{(\eps_L-f_\dis)|\La\cap\Ext\Gamma_{\text{large}}|}
e^{\|\partial\Ext\Gamma_{\text{large}}\|}
\prod_{\gamma\in\Gamma_{\text{large}}}
e^{-({\frac\beta 8}-2)\|\gamma\|} e^{(\eps_L-f)|\Int\gamma|}
\\
&=e^{(\eps_L-f)|\La|+\|\partial\La\|}
\sum_{\Gamma_{\text{large}}}
e^{-a_\dis|\La\cap\Ext\Gamma_{\text{large}}|}
\prod_{\gamma\in\Gamma_{\text{large}}}
e^{-({\frac\beta 8}-3)\|\gamma\|}\,.
\end{aligned}
\]
In order to prove statement (ii), we will have to show that
\begin{equation}
\label{main-ind-bd}
\sum_{\Gamma_{\text{large}}}
e^{-\frac 12 a_\dis|\La\cap\Ext\Gamma_{\text{large}}|}
\prod_{\gamma\in\Gamma_{\text{large}}}
e^{-({\frac\beta 8}-3- \frac{c\beta}{2})\|\gamma\|}
\leq e^{\|\partial\La\|}\,.
\end{equation}
To this end, we define
\[
\tilde K(\gamma)
=\begin{cases}
e^{-({\frac\beta 8}-4- \frac{c\beta}{2})\|\gamma\|}
&\text{if $\gamma$ is a large contour with external label $\dis$}
\\
0&\text{otherwise,}
\end{cases}
\]
and
\[
\tilde Z(\La')
=
\sum_{\Gamma}\prod_{\gamma\in\Gamma} \tilde K(\gamma)\,,
\]
where the sum  runs over sets $\Gamma$ of pairwise compatible
contours in $\La'$ which all have external label $\dis$.  We
also define
\[
\tilde f=-\frac 1{L^d}\log\tilde Z(V_{L,d}).
\]

{We will need the following lemma, whose proof we defer to Appendix~\ref{sec:a-lemma}\,.
\begin{lemma}
\label{lem:a-lem}
Let $\Lambda'$ be of the form \eqref{form-La}.  Then
\begin{equation}
\label{tilde-Z-bd}
\Bigl|\log\tilde Z(\La')+\tilde f|\La'|\Bigr|
\leq \|\partial\La'\|.
\end{equation}
Furthermore $-\tilde f=|\tilde f|\leq \frac {a_\dis}2$.
\end{lemma}
}

{Recalling the definition of
$\tilde{K}(\gamma)$, we now use Lemma~\ref{lem:a-lem}}
to bound the left hand side of
\eqref{main-ind-bd}
 by \\

{
$\displaystyle \sum_{\Gamma_\ext} e^{\tilde
f|\La\cap\Ext\Gamma_\ext|} \prod_{\gamma\in\Gamma_\ext}
e^{-\|\gamma\|}\tilde K(\gamma)$
\begin{eqnarray*}
& \leq &
\sum_{\Gamma_\ext}
e^{\tilde f|\La\cap\Ext\Gamma_\ext|}
\prod_{\gamma\in\Gamma_\ext}
\tilde K(\gamma) e^{\tilde f|\Int\gamma|}\tilde Z(\Int\gamma)\,,  \ \mbox{ (applying \eqref{tilde-Z-bd} to $\Int\gamma$)}
\\
&=&
e^{\tilde f|\La|}
\sum_{\Gamma_\ext}
\prod_{\gamma\in\Gamma_\ext}
\tilde K(\gamma) \tilde Z(\Int\gamma)\\
&=& e^{\tilde f|\La|}\tilde Z(\La)
\leq e^{\|\partial\La\|}, \ \ \mbox{ (once again by \eqref{tilde-Z-bd})\,,}
\end{eqnarray*}
proving \eqref{main-ind-bd}.  Note that in the above,  the sums
run over sets of mutually external contours, all of which have
external label $\dis$.}

 This concludes the proof of (ii) for sets $\La$ of level $k+2$ and
 $\ell=\dis$.  The proof of (ii) for $\ell=\ord$ is identical.

\end{proof}

\subsection{Proof of Lemma~\ref{lem:a-lem}}
\label{sec:a-lemma}
{As before, one can use the forest structure
of sets of pairwise compatible contours to rewrite $\tilde
Z(\La')$ as a sum over sets of mutually external contours in
$\La'$:
\[
\tilde Z(\La')
=
\sum_{\Gamma_\ext}\prod_{\gamma\in\Gamma_\ext} \tilde K(\gamma)\tilde Z(\Int\gamma).
\]
For $C_1$  sufficiently large, the partition function $\tilde
Z(\La')$ can again be analyzed by convergent Mayer expansion,
leading to the bound \eqref{tilde-Z-bd}
(the term proportional to $\eps_L$ is absent since we defined
$\tilde f$ without taking the limit $L\to\infty$).

To bound $\tilde{f}$, we use that
\[ e^{- \tilde{f} L^d}  = \tilde Z(V_{L,d}) \ \le \ \sum_{n=0}^{\infty}
\frac{1}{n!} \sum_{(\gamma_1, \ldots , \gamma_n)} \prod_{i=1}^n \tilde K(\gamma_i)
\ = \
\exp \bigl(\sum_\gamma \tilde K(\gamma) \bigr)\,,\]
where the first sum goes over (not necessarily compatible)
sequences of contours in $V_{L,d}$, with external label $\dis$.
To bound the sum in the exponent, we use that $\tilde
K(\gamma)=0$, unless $\gamma$ is large, which by
\eqref{g-larger-diam-g} implies that $\|\gamma\| \ge k_0 =
2c\beta/a_{\dis}$.
{Furthermore, if $\gamma$ is a contour in $V_{L,d}$ it must be incompatible with
one of the $L^d$ contours obtained by considering configurations with one disordered edge in the $1$-direction. Since
all of these have size $2$, we may use
Lemma~\ref{lem:cont-estimates} to bound the number of contours of size $\|\gamma\|=k$ by
$2L^d (Cd)^k$.  Using these two observations, we get}
\begin{eqnarray*}
\sum_{\gamma \ {\rm large}} \tilde K(\gamma)  &\le &
{2L^d}\sum_{\gamma : \|\gamma\|=k\ge k_0}
(Cd)^k e^{-({\frac\beta 8}-4- \frac{c\beta}{2}) k}\\
&\le & {2L^d}\sum_{k\ge k_0} \Bigl[(Cd) e^{-({\frac\beta 8}-4- \frac{c\beta}{2})} \Bigr]^k\\
&\le & {2L^d}\sum_{k\ge k_0} \Bigl[ \frac{1}{4} \ e^{- c\beta} \Bigr]^k
\le {L^d}e^{-{c\beta k_0}}\,,
\end{eqnarray*}
again provided that $C_1$ is sufficiently large.} Thus
we have that
\[
-\tilde f=|\tilde f|\leq {e^{-{c\beta k_0}}\leq  \frac{1}{c\beta k_0}} = \frac{a_\dis}{{2}(c\beta)^2}
\leq \frac {a_\dis}2.
\]

\subsection{Proof of Lemma~\ref{lem:PS} (iii)}


We start with the observation that the weights
$K_\ell(\gamma)$, and hence the weights $K_\ell'(\gamma)$ are
continuous functions of $\beta$.  Since the free energies
$f_\ell$ are given in terms of an absolutely convergent power
series in the weights $K'_\ell(\gamma)$, they are continuous
functions of $\beta$ as well. Taking into account this
continuity, the following lemma immediately implies
Lemma~\ref{lem:PS} (iii).  {Recall the definition of
$M(\beta)$ from the introduction -- see below \eqref{mu1bd}.}

\begin{lemma}
\label{lem:PS-A2} Assume that \eqref{kappa-asum2} holds.

(i) If $a_\ord=0$, then $M(\beta)>0$.

(ii) If $a_\ord>0$, then $M(\beta)=0$.
\end{lemma}

\begin{proof}
At this point, the proof of Lemma~\ref{lem:PS-A2} is pretty
standard. We therefore only sketch the main steps.

First, we note that for $\LA=\LA_L=\{1,\dots,L\}^d$, the
representations \eqref{FKmes} and \eqref{ESmes} can be
generalized to the model with $1$-boundary conditions defined
in \eqref{mu1bd}. Indeed, let $G_+$ be the induced graph on
$\LA_+=\{0,1,\dots,L+1\}\subset \bbZ^d$. The Edwards-Sokal
measure $\pi_{\Lambda,1}$ corresponding to $\mu_{\Lambda,1}$
can then be obtained from the measure $\pi_{G_+}$ by
conditioning on $\sigma_x=1$ for all $x\in\LA_+\setminus\LA$
and $xy\in A$ whenever $\{x,y\}\subset \LA_+\setminus\LA$.
Next, we observe that in the conditional measure
$\pi_{G_+}(\cdot\mid A)$, a spin at a vertex $x\in\Lambda$ has
probability $1/q$ of taking the value $1$ unless $x$ lies in
the same component of $(\LA_+,A)$ as $\LA_+\setminus\LA$, in
other words, unless $x\in\Ext(A)$.  Keeping these two
observations in mind, the derivation of the representation
\eqref{Z-ord-alt} can easily be adapted to obtain a contour
representation for the magnetization.  Setting $\La=(-\frac
14,L+\frac 54)^d\subset\bbR^d$, this gives
\[
M_\LA(\beta)=\Bigl(1-\frac 1q\Bigr)
\frac {1}{Z_\ord(\La)}\,
e^{-e_\ord|\LA_+|}
\sum_{\Gamma}\frac{ |\LA \cap\Ext\Gamma|}{|\LA|}
\prod_{\gamma\in\Gamma}K_{\ord}(\gamma),
\]
where the sum goes over sets of pairwise compatible contours in
$\La$ with external label $\ord$.  Note that we have chosen
$\La$ in such a way that all edges in $\LA_+\setminus\LA$ lie
in $\Ext\gamma$ whenever $\gamma$ is a contour in $\La$,
corresponding to the above conditioning in $\pi_{G_+}$.

If $a_\ord=0$, the weights $K_\ord(\gamma)$ are bounded by
$e^{-c\beta\|\gamma\|}$ for all $\gamma$.  As a consequence, we
may use a standard Peierls argument to show that the
probability that a given point $x\in\LA$ lies \emph{not} in
$\Ext\Gamma$ is small uniformly in $L$ and $x\in\LA_L$,
implying that $M(\beta)>0$, which proves (i).

Assume finally that $a_\ord>0$ (which implies in particular
that $a_\dis=0$). We will show  that with probability tending
to one, $|\LA\cap\Ext\Gamma|\leq L^{d-\eps}$. Since the ratio
$|\LA\cap\Ext\Gamma|/|\LA|$ in the definition of the
magnetization is bounded uniformly in $L$, this will show that
$M_\LA(\beta)\to 0$ as $L\to\infty$.

Recall the definition of large contours from the last proof,
and let $\Gamma_{\text{large}}$ be the set of external contours
in $\Gamma$ which are large. Then $|\LA\cap\Ext\Gamma|\leq
|\LA\cap\Ext\Gamma_{\text{large}}|$, implying that it will be
enough to show that  the probability that
$|\LA\cap\Ext\Gamma_{\text{large}}|\geq L^{d-1/2}$ goes to zero
as $L\to\infty$. To bound this probability, we will prove an
upper bound on the sum over contours with
$|\LA\cap\Ext\Gamma_{\text{large}}|\geq L^{d-1/2}$ and a lower
bound on $Z_\ord(\La)$.

To obtain the desired upper bound, we proceed as in the proof
of Lemma~\ref{lem:PS-A1} (ii), leading to the estimate
\begin{equation}
\label{M-bd}
\begin{aligned}
e^{-e_\ord|\LA_+|}
&\sum_{\Gamma:
\atop
|\LA\cap\Ext\Gamma_{\text{large}}|> L^{d-1/2}}
\prod_{\gamma\in\Gamma}K_{\ord}(\gamma)\leq
\\
&\leq
e^{-f|\LA_+|}e^{2\|\partial\La\|}
\max_{\Gamma_\ext:
\atop
|\La\cap\Ext\Gamma_\ext|>L^{d-1/2}}
e^{-\frac {a_\ell}2|\La\cap\Ext\Gamma_\ext|}
\prod_{\gamma\in\Gamma_\ext}
e^{-\frac c2\beta\|\gamma\|}
\\
&\leq
e^{-f|\LA_+|}e^{2\|\partial\La\|}
e^{-\frac{ a_\ord}2 L^{d-1/2}}\,,
\end{aligned}
\end{equation}
where the maximum in the second to last line goes over sets of
mutually external contours in $\La$ which all have external
label $\ord$.

To bound $Z_\ord(\La)$ from below we restrict the sum in
\eqref{Z-ord-alt0} to a single term, the term
$\Gamma_\ext=\{\gamma_0\}$, where $\gamma_0$ is the contour
$\gamma_0=\partial [1/4,L+3/4]^d$.  This gives
\[
Z_\ord(\La)
\geq
e^{-e_\ord|\LA_+\setminus\LA|}
e^{-\kappa\|\gamma_0\|}
Z_\dis(\Int\gamma_0)
\geq e^{-fL^d} e^{-2d(\kappa+1+O(e^{-\beta}))L^{d-1}},
\]
where we used the bound \eqref{Zell-low-bd} and the fact that
$a_\dis=0$ in the second step. After extracting the leading
contribution $e^{-fL^d}$, the right hand side falls at most
like an exponential in $L^{d-1}$, while the corresponding decay
in \eqref{M-bd} is exponential in $L^{d-1/2}$.  This proves
that the ratio of \eqref{M-bd} and $Z_\ord(\La)$ goes to zero
as $L\to\infty$, as desired, completing the proof of the lemma.
\end{proof}

\end{document}